\documentclass{itogi2016}
\usepackage[T2A]{fontenc}
 \currentyear{}
 \currentvolume{}

\usepackage[russian]{babel}
\usepackage[cp1251]{inputenc}

\usepackage{graphicx}
\usepackage{enumerate}
\usepackage{amssymb}
\usepackage[pdftex,unicode,colorlinks,linkcolor=blue,
citecolor=red,bookmarksopen,pdfhighlight=/N]{hyperref}

\newtheorem{thm}{Теорема}[section]

\newtheorem{Thopi}{Теорема о повторных интегралах}
\newtheorem{lemma}{Лемма}[section]
\newtheorem{propos}{Предложение}[section]
\newtheorem{corollary}{Следствие}[section]

\theoremstyle{definition}
\newtheorem{definition}{Определение}
\newtheorem{remark}{Замечание}[section]
\newtheorem{example}{Пример}[section]

\renewcommand{\Re}{{\rm Re \,}}
\renewcommand{\leq}{\leqslant}
\renewcommand{\geq}{\geqslant}
\newcommand{\mc}{\mathcal}
\newcommand{\e}{\varepsilon}

\newcommand{\rad}{\text{\tiny\rm rad}}
\newcommand{\bal}{\rm {bal}}
\def\Re{\operatorname{Re}}
\def\Im{\operatorname{Im}}

\def\RR{\mathbb R}
\def\CC{\mathbb C}
\def\NN{\mathbb N}
\def\DD{\mathbb D}

\DeclareMathOperator{\clos}{clos}

\DeclareMathOperator{\Har}{har}
\DeclareMathOperator{\Hol}{Hol}
\DeclareMathOperator{\Poi}{P}

\DeclareMathOperator{\Arg}{Arg}

\DeclareMathOperator{\Exp}{Exp}
\DeclareMathOperator{\Zero}{Zero}
\DeclareMathOperator{\sbh}{sbh}

\DeclareMathOperator{\supp}{supp}

\DeclareMathOperator{\type}{type}
\DeclareMathOperator{\ord}{ord}

\DeclareMathOperator{\up}{\text{\rm \tiny up}}
\DeclareMathOperator{\lw}{\text{\rm \tiny low}}

\DeclareMathOperator{\dd}{\,{\mathrm d\!}}
\renewcommand{\theequation}{\thesection.\arabic{equation}}
\begin{document}

\title[Классические выметания мер и субгармонических функций на систему лучей]{Классические выметания мер\\ и субгармонических функций  на систему лучей}

\author[Хабибуллин Булат Нурмиевич]{Б. Н. Хабибуллин}
\address{Башкирский государственный университет}
\email{khabib-bulat@mail.ru}

\author[Хасанова Альбина Валиевна]{А. В. Хасанова}
\address{Башкирский государственный университет}
\email{albina\_xa@mail.ru}





\UDC{517.53, 517.574}

\begin{abstract}
Конструируются и применяются классические выметания мер и  субгармонических функций на замкнутую систему лучей   комплексной плоскости с вершиной в нуле, включая меры и субгармонические функции и бесконечного порядка. Необходимость такой процедуры возникает при исследовании поведения целых и субгармонических функций на системах лучей. Результаты применяются к полной регулярности роста субгармонических и целых функций на системе лучей, исследованию распределения нулей  целых функций экспоненциального типа класса A. Затрагивается неполнота экспоненциальных систем. Работа содержит как разделы обзорного характера, так и новые результаты.
\end{abstract}

\thanks{Работа выполнена при финансовой поддержке РФФИ (грант №\,16-01-00024).}


\maketitle

\tableofcontents

\section{Введение}\label{s10}

\subsection{Истоки и предмет статьи}\label{111}

В контрасте со многими классическими именными математическими результатами и методами, приписываемое авторство которым  часто бывает  сомнительным, 
можно с уверенностью сказать, что первооткрывателем техники выметания\footnote{balayage or sweeping (out) method} меры и потенциала  на границу области, или из области, был Анри Пуанкаре  (конец XIX века \cite{P1}--\cite{P2}). В первой половине XX века его метод классического выметания был существенно развит  Шарлем Валле-Пуссеном \cite{V-P}. Дальнейший необъятный список авторов и работ, причастных к развитию метода выметания, можно частично извлечь из монографий Н.\,О.~Ландкофа \cite{L}, Ю.~Блайднера и В.~Хансена
\cite{BH}. В последней книге большое внимание уделено абстрактной форме выметания. Несколько иная общая форма выметания разрабатывалась и широко использовалась нами применительно к разнообразным задачам
в \cite{Kha01}--\cite{Khab011}. В настоящей работе мы возвращаемся к классическому выметанию  на конкретные замкнутые множества: системы лучей с одной общей  вершиной в нуле. Специфика таких множеств --- некомпактность, возможная бесконечность числа неограниченных компонент связности, неограниченность выметаемых на систему лучей мер и и субгармонических функций, вплоть до бесконечного порядка. Ранее один из авторов уже рассматривал эти вопросы в связи с асимптотическим поведением субгармонических и целых функций и распределением их мер Рисса и последовательностей нулей в    \cite{Kha88}--\cite{Khsur}. При этом был получен ряд новых результатов: критерий вполне регулярного роста целой функции на системе  лучей,
перенос теоремы Рубела\,--\,Мальявена \cite{MR}, \cite[гл.~22]{RC}
о минимально возможной ширине индикаторной диаграммы целой функции экспоненциального типа с положительными нулями на функции с произвольными комплексными нулями, законченный критерий полноты экспоненциальной системы в пространстве голоморфных функций в произвольной фиксированной выпуклой неограниченной области в естественных топологиях и исключительно в геометрических терминах, а также  многое другое. При этом для тонких ситуаций \cite{kh91AA} использовалось не только классическое выметание, но и новое на тот период выметание рода 1, когда выметание меры дает заряд, т.\,е. вещественнозначную меру. Выметания рода $1$ \cite{kh91AA}--\cite{KhDD92} или, более общ\'о, рода $q\in \NN$ предполагается дополнительно исследовать в иной работе. Поэтому некоторая подготовительная работа <<с запасом>> для этих целей проведена и в настоящей статье там, где это уместно и вписывается  в непосредственно применяемые частные случаи.
 
Здесь мы исследуем специальные типы выметания мер и субгармонических функций на системы лучей на комплексной плоскости $\CC$ с вершиной в нуле, которые, хотя и представляют частные случаи выметаний, весьма полезны при исследовании роста и асимптотического поведения целых и субгармонических функций на системах лучей из  $\CC$. Основные обозначения и определения введены в подразделе \ref{s1.1}, к которому можно обращаться по мере необходимости.  Во вспомогательных целях систематизируются различные известные шкалы роста около бесконечности (исчезания около нуля) функций и мер в разделе \ref{sss:sh}.

Существенно совершенствуется, особенно в количественном аспекте,  классическое выметание на систему лучей за счет в основном точных верхних и нижних оценок гармонической меры для полуплоскости (угла) подразделах \ref{p_3_1}, \ref{s:dal}. При выметании меры и субгармонической функции рассматриваем не только меры и функции конечного порядка, но и меры и функции бесконечного порядка в разделах   \ref{bal_S_0}--\ref{imu}.  Основной результат о распределении нулей  функциях вполне регулярного роста приобретает законченный характер в подразделе  \ref{vpreg}. В подразделе \ref{bAmp} показано, что за счет функции-мультипликатора индикаторную диаграмму  целой функция экспоненциального типа и класса A можно сделать сколь  угодно узким. Последнее применяется к неполноте экспоненциальных систем в вертикальных сколь угодно узких  (полу)полосах.

 В отношении целых и субгармонических функций опираемся прежде всего на терминологию, обозначения и понятия из \cite{Levin56}, \cite{Levin96}, \cite{HK}, \cite{Az}.  \cite{Rans}.

\subsection{Основные обозначения, определения и соглашения}\label{s1.1}
\subsubsection{Множества, порядок, топология}\label{1_1_1}
Как обычно,  $\NN=\{1,2,\dots\}$ --- множество {\it натуральных чисел,\/} но $\NN_0:=\{0\}\cup \NN$ --- <<французское>> множество натуральных чисел,  $\RR$ и  $\CC$  --- множества всех {\it  вещественных,\/} или действительных,  и  {\it комплексных чисел\/} соответственно. 
Эти числовые множества, если необходимо и возможно,  рассматриваются  в их естественной алгебраической, порядковой, геометрической и/или топологической интерпретации. 
Через $\CC_{\infty}:=\CC\cup\{\infty\}$ обозначаем {\it расширенную комплексную плоскость,\/} или сферу Римана с естественной топологией. Здесь мы предпочитаем избегать использования сферического расстояния и для $z\in \CC_{\infty}$ через $|z|$ обозначаем обычный модуль числа $z\in \CC$, а  $|\infty|:=+\infty$. Используем обозначения $\CC^{\up}:=
\{z\in \CC \colon \Im z>0\}$ и $\CC_{\lw}:=-\CC^{\up}=\{z\in \CC\colon -z\in \CC^{\up}\}$  --- {\it верхняя и нижняя полуплоскости;\/}
$\overline \CC^{\up}:=\CC^{\up}\cup \RR$ и $\overline \CC_{\lw}:=\CC_{\lw}\cup \RR$ --- {\it замкнутые в $\CC$ верхняя и нижняя полуплоскости}. Используем также обозначения 
	\begin{subequations}\label{se:R}
	\begin{align} 
	\RR_{-\infty}:=\{-\infty\}\cup \RR=:[-\infty,+\infty),& \quad \RR_{+\infty}:=\RR\cup \{+\infty\}=:(-\infty,+\infty],
		\tag{\ref{se:R}a}\label{seR:a}
	\\ 
	\RR_{\pm\infty}:=\RR_{-\infty}\cup\RR_{+\infty}=:[-\infty, +\infty],& \quad\text{а также\;  $\RR_{\infty}=\RR\cup \{\infty\}\subset \CC_{\infty}$},  
	\tag{\ref{se:R}b}\label{seR:b}
	\end{align}
	\end{subequations}
 где в \eqref{seR:b} определены  соответственно {\it расширенная вещественная  прямая\/} с естественными  отношением порядка $-\infty \leq x\leq +\infty$ для любого $x\in \RR_{\pm\infty}$ и порядковой топологией, а также {\it замыкание $\RR$ в $\CC_{\infty}$}.
Таким образом, открытые окрестности двух  точек $-\infty \in \RR_{\pm\infty}$ и $+\infty \in \RR_{\pm\infty}$  --- это, наряду с $\RR_{\pm\infty}$, соответственно все множества вида $(a,+\infty]:=\{x\in \RR_{\pm\infty}\colon a<x\leq +\infty\}$ и вида $[-\infty,a):=\{x\in \RR_{\pm\infty}\colon -\infty\leq x<a\}$ с произвольным $a\in \RR$.
 {\it Интервал} --- связное подмножество в $\RR_{\pm \infty}$. В то же время открытые окрестности точки $\infty\in \RR_{\infty}$ --- это множества вида $\{x\in \RR_{\infty}\colon |x|\geq a\}$ с произвольным $a\in \RR$.
На подмножества $S\subset \CC_{\infty}$ индуцируется топология с $\CC_{\infty}$, но при рассмотрении  $S\subset \RR_{\pm\infty}$
как подмножества в $\RR_{\pm\infty}$ уже с $\RR_{\pm\infty}$. Если $S\subset \CC_{\infty}$, то 
\begin{equation}\label{procS:0}
S_*:=S\setminus \{0\}\quad \text{--- {\it<<проколотое>> в нуле\/} $S$}.
\end{equation}

 Для  $z\in \CC$ через  $\bar z$ обозначаем  сопряженное комплексное число, $\bar{\infty}:=\infty$. Для $S\subset \CC_{\infty}$ полагаем $\bar S:=\{\bar z \colon z\in S\}$.
 
Одним и тем же символом $0$ обозначаем, по контексту, число нуль, нулевой вектор, нулевую функцию, нулевую меру (заряд) и т.\,п.;
 $\varnothing$ --- {\it пустое  множество.\/} 
 Для подмножества  $X$ упорядоченного векторного пространства  чисел, функций, мер и т.\,п. с отношением порядка \;$\geq$\;
полагаем   	$X^+:=\{x\in X\colon x \geq 0\}$
 --- все положительные элементы из $X$; $x^+:=\max \{0,x\}$. 
{\it Положительность\/} всюду понимается, в соответствии с контекстом, как $\geq 0$.  
{\it На множествах функций\/}   с упорядоченным  множеством значений {\it отношение  порядка\/} индуцируется с множества значений как  {\it поточечное.} 

Для подмножеств $S_0\subset S\subset T$ топологического пространства $T$ через $\clos_S S_0$, $\partial_S S_0$
обозначаем 
соответственно {\it замыкание\/} множества $S_0$ в $S$ и его {\it границу\/} в $S$ в индуцированной с $T$ топологии  в $S$. Для $S_0 \subset S\subset T$ пишем $S_0\Subset S$, если $S_0$ --- {\it относительно компактное подмножество\/} в $S$ в топологии, индуцированной с $T$ на $S$. В основном в настоящей работе  $T=\CC_{\infty}$. 

Для  $r\in \RR_{\pm\infty}$ и $z\in \CC$ полагаем $	D(z,r):=\{z' \in \CC \colon |z'-z|<r\}$  ---  {\it открытый  круг с центром\/ $z$ радиуса\/ $r$.} Так, $D(z,r)=\varnothing$ при $r\leq  0$ и $D(z,+\infty):= \CC$. В частном случае  $D(r):=D(0,r)$; $\DD:=D(1)$ --- {\it единичный круг.\/} 
Кроме того, $\overline{D}(z,r):=\clos D(z,r)$ --- {\it замкнутый круг с центром\/ $z$ радиуса\/ $r\neq 0$.}  Так, $\overline{D}(z,+\infty):=\CC_{\infty}$,
но $\overline{D}(z,0):=\{z\}$.

\subsubsection{ Функции}\label{1_1_2}  Для {\it характеристической функции произвольного  подмножества\/} $A\subset X$
 на $X$ используем обозначение
 \begin{equation}\label{chX}
 \boldsymbol{1}_A(x):=\begin{cases}
 1&\quad\text{при $x\in A$},\\
 0&\quad\text{при $x\in X\setminus A$},
  \end{cases}
 \end{equation}
 как правило,  не прописывая  явно множество $X$.

Для топологического пространства $T$, как обычно, $C_{\RR}(T)$ и С(Е)--- векторные пространства соответственно над $\RR$ со значениями из $\RR$ и над  $\CC$ со значениями в $\CC$. 
	
	Для открытого множества   $\mathcal O \subset \CC_{\infty}$ через  $\Hol ({\mathcal O})$  и $\Har ({\mathcal O})$ обозначаем  векторные  пространства соответственно над полем $\CC$ {\it голоморфных\/} и над полем $\RR$ {\it гармонических\/} в ${\mathcal O}$ функций; $\sbh ({\mathcal O})$ --- выпуклый конус над $\RR^+$ {\it субгармонических\/} в ${\mathcal O}$ функций \cite{HK}, \cite{Rans}.  Субгармоническую функцию, {\it тождественно равную\/} $-\infty$ на ${\mathcal O}$, обозначаем символом $\boldsymbol{-\infty}\in \sbh ({\mathcal O})$; $\sbh_*(\mathcal O):=\sbh (\mathcal O)\setminus \{\boldsymbol{-\infty}\} $. 
Для функции $f$, определенной на множестве с подмножеством $X$ {\it сужение\/ $f$ на\/}  $X$  обозначаем как $f\bigm|_X$.

\subsubsection{ Меры и заряды}\label{1_1_3}  Далее $\mathcal M (S)$ --- класс всех {\it  счетно-аддитивных функций борелевских подмножеств борелевского множества $S\subset \CC_{\infty}$ со значениями в $\RR_{\pm\infty}$, конечных на компактах из $S$.\/} Элементы  из $\mathcal M (S)$ называем {\it зарядами,\/} или   {\it вещественными мерами,\/}    на   $S\subset \CC_{\infty}$;     $\mathcal M^+ (S):=\bigl(\mathcal M (S)\bigr)^+$ --- подкласс {\it положительных мер;\/} $\supp \mu$ --- {\it носитель заряда/меры} $\mu\in \mathcal M (S) $. Заряд $\mu \in  \mathcal M (S)$ {\it сосредоточен\/} на измеримом по мере $\mu$ подмножестве $S_0\subset S$, если $\mu (S')=\mu (S'\cap S_0)$ для любого измеримого по $\mu$ подмножества $S'\subset S$. Для измеримого по $\mu\in \mathcal M(S)$ подмножества  $S_0\subset S$ через $\mu\bigm|_{S_0}$ обозначаем {\it сужение заряда $\mu$ на $S_0$.\/} Через $\lambda_{\CC}$ и $\lambda_{\RR}$ обозначаем {\it лебеговы меры\/} соответственно на 
$\CC_{\infty}$ и на $\RR_{\infty}$ с $\lambda_{\CC}\bigl(\{\infty\}\bigr):=0$ и $\lambda_{\RR}\bigl(\{\pm\infty\}\bigr):=0$, а также их сужения на подмножества из $\CC_{\infty}$. Нижний индекс $\CC$ в $\lambda_{\CC}$ при этом часто опускаем.  Через $\delta_z$ обозначаем {\it меру Дирака в точке $z\in \CC_{\infty}$,\/} т.\,е. вероятностную меру с $\supp \delta_z=\{z\}$, $\delta_z(A)\overset{\eqref{chX}}{:=}
\boldsymbol{1}_A(z)$. 
Как обычно, для $\mu\in \mathcal M(\CC)$ через $\mu^+:=\max\{0,\mu\}$, $\mu^-:=(-\mu)^+$ и $|\mu|:=\mu^++\mu^-$,
обозначаем {\it верхнюю, нижнюю\/  {\rm и} полную вариации заряда\/ $\mu=\mu^+ -\mu^-$.} 
Для $\nu\in \mathcal M (S )$ и $D (z,r)\subset S$ полагаем 
\begin{equation}\label{df:nup}
\nu (z,r):=\nu \bigl(\,\overline D(z,r)\bigr),\quad \nu^{\rad}(r):=\nu(0,r) 
\end{equation}
--- {\it считающая  функция с центром\/ $z$} и {\it считающая радиальная функция} заряда $\nu$ соответственно. Кроме того, для  заряда  $\nu\in \mathcal M(\CC)$  используем {\it функцию\/} $\nu^{\RR}$ его {\it распределения на\/} $\RR$, определяемую по   правилу
\begin{equation}\label{nuR}
\nu^{\RR}(t):=\begin{cases}	
-\nu\bigl( [t,0) \bigr) \;  &\text{ при } t<0 ,\\
\nu\bigl( [0,t] \bigr)  \;  &\text{ при } t\geq 0.
\end{cases}
\end{equation}

Меру Рисса функции ${u}\in \sbh ({\mathcal O})$   обозначаем как $\nu_{u}:=\frac{1}{2\pi}\Delta {u}\in \mathcal M ^+(\mathcal O)$ или $\mu_u$ и  т.\,п. 
Так, для  ${u}\neq \boldsymbol{-\infty}$ её мера Рисса  $\nu_{u}$ --- борелевская положительная мера \cite{Rans}, \cite{HK}. Для  функции же $\boldsymbol{-\infty}\in \sbh({\mathcal O})$ её мера Рисса  по определению равна $+\infty$ на любом подмножестве из ${\mathcal O}$.

\subsubsection{ Нули голоморфных функций}\label{zf}
Пусть $0\neq f\in \Hol({\mathcal O})$. Функция $f$ {\it обращается в нуль на последовательности точек\/ 
${\tt Z}=\{{\tt z}_k\}_{k=1,2,\dots}$,}
 лежащих в $\mathcal O$ (пишем $\tt Z\subset \mathcal O$), если кратность нуля, или корня,  функции $f$ в каждой точке $z \in \mathcal O$ не меньше числа повторений этой точки  в последовательности ${\tt Z}$ (пишем $f({\tt Z})=0$).
Каждой последовательности $\tt Z\subset \mathcal O$ без точек сгущения в $\Omega$ сопоставляем её {\it считающую меру\/}
\begin{equation}\label{df:nZS}
	n_{\tt Z}(S):=\sum_{{\tt z_k}\in S} 1 \quad \text{\it для  произвольных $S\subset D$}
\end{equation}
 --- число точек из $\tt Z$, попавших в $S$.
Последовательность нулей (корней) ненулевой функции $f\in \Hol (D)$, каким-либо образом перенумерованную с учетом кратности, обозначаем через $\Zero_f $. Так как  $\log |f|\in \sbh (\mathcal O)$,  взаимосвязь меры Рисса $\nu_{\log |f|}$   со считающей мерой  её нулей \eqref{df:nZS} задаётся равенством \cite[{теорема }3.7.8]{Rans}
\begin{equation}\label{nufZ}
	\nu_{\log |f|}=\frac1{2\pi} \Delta \log |f|\overset{\eqref{df:nZS}}{=}n_{\Zero_f}.
\end{equation}
В частности, из $f({\tt Z})=0$ следует $n_{\tt Z}\leq n_{\Zero_f}$ на $\mathcal O$, и наоборот.

\subsubsection{ Некоторые соглашения}\label{ns} Число $C$ и постоянную функцию, тождественно равную  $C$, не различаем.  Ссылка над знаками (не)равенства, включения, или, более общ\'о, бинарного отношения, означает, что при переходе к правой части этого отношения применялись, в частности, и эта  формула, определение, обозначение или  утверждение. 

{\it (Под)область в\/ $\CC_{\infty}$ --- открытое связное подмножество в\/ $\CC_{\infty}$.}

\section{Шкалы роста и исчезания функции}\label{sss:sh} 
\setcounter{equation}{0}
 Основная цель этого  \S~\ref{sss:sh} --- систематизировать в целом известные элементарные факты, изложения и применения которых разбросаны по различным источникам, а также согласовать определения и понятия. Определения и утверждения  этого подраздела~\ref{sss:sh}  не зависят от выбора  числа $r_0>0$.
\subsection{Рост около\/ бесконечности} 
Для  $f\colon [r_0,+\infty) \to \RR$ определим 
\begin{subequations}\label{se:nu0}
\begin{align} 
\ord_{\infty}[f]&:=\limsup_{r\to +\infty} 	\frac{\log^+f^+(r)}{\log r}\in \RR^{+}
\tag{\ref{se:nu0}a}\label{senu0:a}
\\
\intertext{{\it --- порядок роста\/} (функции)  $f$ {\it около\/} (точки) $+\infty$; для $p\in \RR^+$} 
	\type^{\infty}_p[f]&:=\limsup_{r\to +\infty} 	\frac{f^+(r)}{r^p}\in \RR^{+}
\tag{\ref{se:nu0}c}\label{senu0:c}
\end{align}
\end{subequations}
\\ 
{\it --- тип  роста\/} (функции) $f$ {\it при порядке\/ $p$ около\/} (точки) $+\infty$.
Очевидно, 
\begin{equation}\label{t_o}
\text{\it если $\type_p^{\infty}[f]<+\infty$, то $\ord_{\infty}[f]\leq p$}.
\end{equation}
Обратная импликация неверна. Используем еще одну характеризацию  роста.  

Функция $f\colon [r_0,+\infty)\to \RR$ 
 принадлежит {\it классу сходимости\/} или {\it классу расходимости при порядке роста\/ $p\in \RR^+$ около\/} (точки) $+\infty$ 
\cite[определение 4.1]{HK}, если соответственно сходится (к конечному числу) или расходится интеграл
\begin{equation}\label{{se:fcc}a}
\int_{r_0}^{+\infty}\frac{f(t)}{t^{p+1}}\dd t.
\end{equation}.
\begin{propos}\label{pr:fcc} Пусть  $f\colon [r_0,+\infty)\to \RR$ ---   возрастающая функция. 
\begin{enumerate}[{\rm (i)}]
	\item\label{i:i1} Из сходимости интеграла \eqref{{se:fcc}a} следует $\type_p^{\infty}[f]=0$  и 
	 сходимость интеграла Стильтьеса
\begin{equation}\label{cc:fint}
\int_{r_0}^{+\infty}\frac{\dd f(t)}{t^{p}}<+\infty.
\end{equation}
\item \label{i:i2} Обратно, из  \eqref{cc:fint}  при $p>0$ следует $\type_p^{\infty}[f]=0$  и сходимость  интеграла \eqref{{se:fcc}a}, а при $p=0$ существует  $\lim\limits_{r\to +\infty}f(r)<+\infty$ и $\type_0^{\infty}[f]<+\infty$. 
\end{enumerate}
Связь между интегралами \eqref{{se:fcc}a} и \eqref{cc:fint} в условиях \eqref{i:i1} и \eqref{i:i2} --- равенство
\begin{equation}\label{in:poch}
	\int_{r}^{+\infty}\frac{\dd f(t)}{t^{p}}=-\frac{f(r)}{r^p}+p\int_{r}^{+\infty}\frac{f(t)}{t^{p+1}}\dd t
\end{equation}
\end{propos}
\begin{proof} Если в \eqref{i:i1}  $p=0$, то из сходимости интеграла в \eqref{{se:fcc}a} следует, что существует 
$\lim\limits_{r\to +\infty}f(r)=0$ и  п.~\eqref{i:i1} заключения очевидны. Значит достаточно рассмотреть случай $p>0$.
Как в \cite[лемма 4.3]{HK} при $r\geq r_0$ имеем
\begin{equation}
f(r) \frac1p\, r^p \leq \int_{r}^{+\infty}\frac{f(t)}{t^{p+1}}\dd t\underset{r\to +\infty}{\longrightarrow}0,
 \end{equation}
что влечет за собой $\type_p^{\infty}[f]=0$. Отсюда, подобно \cite[(4.2.5)--(4.2.6)]{HK}, интегрируя по частям, для любого 
$r\geq r_0$ получим \eqref{in:poch},
где  сходимость последнего интеграла влечет за собой \eqref{cc:fint}.

Обратно, если $p>0$, то при $r_0\leq r<R$ из \eqref{cc:fint} 
\begin{equation*}
\frac{1}{R^p}\bigl(f(R)-f(r)\bigr)\leq \int_{r}^{R}\frac{\dd f(t)}{t^{p}}\dd t\underset{r\to +\infty}{\longrightarrow}0,
\end{equation*}
откуда $\type_p^{\infty}[f]=0$. Тогда справедливо \eqref{in:poch} и сходится интеграл \eqref{{se:fcc}a}. Если же $p=0$, то 
\eqref{cc:fint} эквивалентно  существованию предела $\lim\limits_{r\to +\infty}f(r)<+\infty$.
\end{proof}

\subsection{Исчезание  около нуля} 
Для  $f\colon (0,r_0] \to \RR$ определим 
\begin{subequations}\label{se:nu00}
\begin{align} 
\ord_{0}[f]&:=\liminf_{0<r\to 0} 	\frac{\log f^+(r)}{\log r}\in \RR^{+}
\tag{\ref{se:nu00}a}\label{senu00:a}
\\
\intertext{{\it --- порядок исчезания\/} (функции)  $f$ {\it около\/} (точки) $0$; для $p\in \RR^+$} 
	\type^{0}_p[f]&:=\limsup_{0<r\to 0} 	\frac{f^+(r)}{r^p}\in \RR^{+}
\tag{\ref{se:nu00}c}\label{senu00:c}
\end{align}
\end{subequations}
\\ 
{\it --- тип  исчезания\/} (функции) $f$ {\it при порядке\/ $p$ около\/} (точки) $0$.

Очевидно, 
\begin{equation}\label{t_o0}
\text{\it если $\type_p^{0}[f]<+\infty$, то $\ord_{0}[f]\geq p$}.
\end{equation}
Обратная импликация неверна. Используем еще одну характеризацию  роста.  

Говорим, что функция $f\colon (0, r_0]\to \RR$ 
 принадлежит {\it классу сходимости\/} или {\it классу расходимости при порядке исчезания\/ $p\in \RR^+$ около\/} (точки) 
$0$, если соответственно сходится (к конечному числу) или расходится интеграл
\begin{equation}\label{{se:fcc}0}
\int_{0}^{r_0}\frac{f(t)}{t^{p+1}}\dd t.
\end{equation}.
\begin{propos}\label{pr:fcc0} Пусть  $f\colon (0,r_0]\to \RR$ ---   возрастающая функция. 
\begin{enumerate}[{\rm (i)}]
	\item\label{i:i10} Из сходимости интеграла \eqref{{se:fcc}0} следует, что  существует 
\begin{equation}\label{ex:limf}
\lim_{0<r\to 0} f(r)\log r =0, 
\end{equation}
а также 
\begin{equation}\label{ex:limf0}
 \int_0^{r_0}\log t \dd f(t)>-\infty,
\end{equation}
а при $p>0$ еще и 	$\type_p^{0}[f]=0$, а также  
\begin{equation}\label{cc:fint0}
\int_{0}^{r_0}\frac{\dd f(t)}{t^{p}}<+\infty.
\end{equation}
\item \label{i:i20} Обратно, если выполнено \eqref{ex:limf0}, то существуют пределы
\begin{equation}\label{es:00l}
	f(0):=\lim_{0<r\to 0} f(r)\in \RR, \quad \lim_{0<r\to 0}\bigl(f(r)-f(0)\bigr)\log r=0, 
\end{equation}
сходится интеграл 
\begin{equation}\label{f-f0}
\int_{0}^{r_0}\frac{f(t)-f(0)}{t}\dd t<+\infty,
\end{equation}
и $\type_0^{0}[f]<+\infty$. Если же при $p>0$ выполнено \eqref{cc:fint0}, то, наряду с \eqref{es:00l}, 
$\type_p^{0}\bigl[f-f(0)\bigr]=0$ и 
\begin{equation}
\int_{0}^{r_0}\frac{f(t)-f(0)}{t^{p+1}}\dd t<+\infty.
\end{equation}
\end{enumerate}
В условиях из \eqref{i:i10} и  \eqref{i:i20} при $p>0$ связь между интегралами \eqref{cc:fint0} и \eqref{f-f0}
дается равенством
\begin{equation}\label{in:poch0}
	\int_{0}^{r_0}\frac{f(t)-f(0)}{t^{p+1}}\dd t=-\frac{f(r_0)-f(0)}{pr_0^p}+\frac1p\int_{0}^{r_0}\frac{\dd f(t)}{t^{p}},
\end{equation}
а при $p=0$ связь между интегралами из \eqref{{se:fcc}0} и \eqref{ex:limf0} --- равенством
\begin{equation}\label{in:0}
	\int_{0}^{r_0}\frac{f(t)-f(0)}{t}\dd t=\bigl(f(r_0)-f(0)\bigr)\log r_0- \int_0^{r_0}\log t \dd f(t).
\end{equation}
\end{propos}
\begin{proof} Из сходимости интеграла \eqref{{se:fcc}0}  $\lim\limits_{0<r\to 0} f(t)=0$. Следовательно, $f$ --- положительная функция и при $0<r<\sqrt r\leq \min \{1,r_0\}$ 
\begin{equation*}
	0\underset{0<r\to 0}{\longleftarrow}\int_r^{\sqrt r}\frac{f(t)}{t^{p+1}}\dd t
	\geq \int_r^{\sqrt r}\frac{f(t)}{t}\dd t\geq -\frac{1}{2}\,f(r)\log r,
\end{equation*}
что доказывает существование предела в \eqref{ex:limf}, равного нулю. Тогда 
\begin{equation*}
	\int_0^{r_0}\log t \dd f(t)=f(r_0)\log r_0-\int_{0}^{r_0}\frac{f(t)}{t}\dd t\geq f(r_0)\log r_0-\int_{0}^{r_0}\frac{f(t)}{t^{p+1}}\dd t>-\infty,
\end{equation*}
что  доказывает \eqref{ex:limf0} и \eqref{in:0} с $f(0)=0$ . При $p>0$
\begin{equation*}
	0\underset{0<r\to 0}{\longleftarrow}\int_{r}^{2r}\frac{f(t)}{t^{p+1}}\dd t\geq f(r)\,\frac{1}{p} \left(\frac{1}{r^p}-\frac{1}{(2r)^p}\right)=\frac{1}{p} \left(1-\frac{1}{2^p}\right)\frac{f(r)}{r^p},
\end{equation*}
откуда $\type_p^{0}[f]=0$.  Тогда можно интегрировать по частям, что дает \eqref{in:poch0} с $f(0)=0$, 
откуда  имеет место \eqref{cc:fint0}.

Обратно, пусть выполнено \eqref{ex:limf0}. Тогда существует число $C\in \RR^+$, с которым для любых $r\in (0, r_1)$ при  
фиксированном $r_1<\min \{1,r_0\}$ 
\begin{equation}\label{Clfr1} 
	-C\leq \int_r^{r_1}\log t \dd f(t)\leq \log r_1 \int_r^{r_1}\dd f(t)=\bigl(f(r_1)-f(r)\bigr)\log r_1, 
\end{equation}
 откуда следует существование первого предела в \eqref{es:00l} и возможность доопределить значение $f(0)\in \RR$. 
В частности, $\type_0^{0}[f]<+\infty$. Кроме того, устремляя в \eqref{Clfr1} сначала $r$ к нулю, а затем и $r_1$, получаем
\begin{equation}\label{r-rf}
	0=\lim_{0<r_1\to 0}\int_0^{r_1}\log t \dd f(t)\leq \liminf_{0<r_1\to 0} \,\bigl(f(r_1)-f(0)\bigr)\log r_1.
\end{equation}
При $0<r<\sqrt r<r_1=\min \{1,r_0\}$ из 
\begin{multline*}
0\underset{0<r\to 0}{\longleftarrow}\int_{r}^{\sqrt r} \log t\dd f(t)=
\int_{r}^{\sqrt r} \log t\dd \bigl(f(t)-f(0)\bigr)\geq \log r\int_{r}^{\sqrt r} \dd \bigl(f(t)-f(0)\bigr)
\\
\geq \log r\int_{0}^{\sqrt r} \dd \bigl(f(t)-f(0)\bigr)=\bigl(f(\sqrt r)-f(0)\bigr)\log r =2\bigl(f(\sqrt r)-f(0)\bigr)\log \sqrt r
\end{multline*}
следует
\begin{equation*}
	0\geq \limsup_{0<r_1\to 0} \,\bigl(f(\sqrt r)-f(0)\bigr)\log \sqrt r,
\end{equation*}
что вместе с \eqref{r-rf} дает второе предельное равенство в \eqref{es:00l}. Теперь можно интегрировать по частям:
\begin{equation*}
\int_{0}^{r_0}\frac{f(t)-f(0)}{t}\dd t=\bigl(f(r_0)-f(0)\bigr)\log r_0- \int_0^{r_0}\log t \dd f(t)\overset{\eqref{ex:limf0}}{<}+\infty,
\end{equation*}
что дает \eqref{f-f0}. Пусть теперь $p>0$. Тогда автоматически выполнено \eqref{ex:limf0}, а  следовательно и \eqref{es:00l}.
При этом из \eqref{cc:fint0} имеем
\begin{equation}
0\underset{0<r\to 0}{\longleftarrow} \int_{0}^{r}\frac{\dd \,\bigl(f(t)-f(0)\bigr)}{t^p}
\geq \frac{f(r)-f(0)}{r^p},
\end{equation}
что доказывает $\type_p^{0}\bigl[f-f(0)\bigr]=0$ и позволяет интегрировать по частям:
\begin{equation*}
\int_{0}^{r_0}\frac{f(t)-f(0)}{t^{p+1}}\dd t=-\frac{f(r_0)-f(0)}{pr_0^p}+\frac1p\int_{0}^{r_0}\frac{\dd f(t)}{t^{p}}
\overset{\eqref{cc:fint0}}{<}+\infty.
\end{equation*}
\end{proof}

\section{Классические выметания}
\setcounter{equation}{0}

\subsection{Гармоническая мера, ее геометрический смысл и интеграл Пу\-а\-с\-с\-о\-на}\label{p_3_1}

 Через $\mathcal B (S)$ обозначаем $\sigma$-алгебру борелевских подмножеств борелевского множества $S\subset\CC_{\infty}$, а при $S\subset \CC$ используем обозначение $\mathcal B_{\rm b} (S)\subset \mathcal B (S)$ для подкласса ограниченных в $\CC$ множеств из $\mathcal B (S)$.
 Для подобласти $D\subset \CC_{\infty}$ с неполярной границей $\partial D\neq \varnothing$ в обозначениях из \cite[определение 4.3.1]{Rans}, несколько отличных от обозначений из \cite[теорема 3.10]{HK}, \cite{GM}, через 
\begin{equation}\label{om_0}
\omega_D\colon D\times \mathcal B (\partial D) \to \RR^+
\end{equation}
обозначаем  функцию, называемую {\it гармонической мерой для\/ $D$},
которая однозначно  определяется требованиями
\begin{enumerate}[{\rm (a)}]
	\item\label{hm:a} $	\omega_D(z, B)\in \mathcal{M}^+(\partial D)$   для каждой точки $z\in D$ и множества $B\in \mathcal B (\partial D)$;
	\item\label{hm:b}  если $f\in C_{\RR}(\partial D)$, то {\it интеграл Пуассона\/} $\mathcal{P}_D f$ функции $f$ на $D$
		\begin{equation}\label{df:PDf}
\mathcal P_Df\colon z\mapsto \int_{\partial D} f(w)\dd \omega_D(z,w), \quad z\in D,
\end{equation}
совпадает с ассоциированной с функцией $f$ {\it функцией Перрона\/}
\begin{equation*}
z\mapsto \sup \Bigl\{u(z)\colon u\in \sbh(D), \limsup_{D\ni z'\to z}u(z')\leq f(z) \text{ для всех $z\in \partial D$}\Bigr\}, \quad z\in D.
\end{equation*}
\end{enumerate}
Очевидно, $\omega_D(z,\cdot)$ --- {\it вероятностная мера,\/} в чем можно убедится на примере $f=1$ на $\partial D$.
Для произвольного подмножества  $B\subset \CC_{\infty}$ при $B\cap \partial D\in \mathcal B (\partial D) $ и точки $z\in D$ полагаем по определению 
\begin{equation}\label{df:o}
\omega_D(z, B):=:\omega(z, B; D):=\omega_D(z, B\cap \partial D)
\end{equation}
--- {\it гармоническая мера множества $B\subset \CC_{\infty}$ в точке $z\in D$ для\/} $D$.

При включениях $\partial D\subset S \subset \CC_{\infty}$ для  произвольной  функции $f\colon S\to [-\infty, +\infty]$ с сужением $f\bigm|_{\partial D}\in C_{\RR}(\partial D)$ {\it интеграл Пуассона\/} $\mathcal P_Df$ определяется равенством
\begin{equation}\label{df:IP}
	\mathcal P_Df:=\mathcal P_D\bigl(f\bigm|_{\partial D}\bigr).
\end{equation}
Определение \eqref{df:IP}, допуская в \eqref{df:IP}  и значения $-\infty$,  можно распространить на полунепрерывные сверху  
сужения $f\bigm|_{\partial D}$  со значениями в $\RR_{-\infty}$, поскольку такая функция --- предел убывающей последовательности функций из $C_{\RR}(\partial D)$. 
Более того, далее в многих ситуациях можно допускать неопределенность функции $f$ в конечном числе точек 
на границе $\partial D$, в роли которых  будут выступать, в основном, две точки $0, \infty \in \CC_{\infty}$. 
В частности, когда $D=\CC^{\up}$,  определения \eqref{om_0}--\eqref{df:IP} можно переписать через  классические {\it ядро  Пуассона\/} 
\begin{equation}\label{df:kP}
	\quad \Poi_{\CC^{\up}}(t,z):=\frac{1}{\pi}\Im \frac{1}{t-z}
	=\frac{1}{\pi}\,\frac{\Im z}{\bigl(t-\Re z\bigr)^2+(\Im z)^2}\, , \quad t\in \RR, \; z\in \CC^{\up},
\end{equation}
а также, для всех $z\in \CC^{\up}$, через {\it интеграл Пуассона\/} 
\begin{subequations}\label{se:om}
 \begin{align} 
 \omega_{\CC^{\up}}(z, B)&=\omega(z, B;\CC^{\up})=\int_{B\cap \RR}\Poi (t,z) \dd t
=\int_{B\cap \RR_{\pm\infty}} \Poi (t,z) \dd \lambda_{\RR}(t),
 \tag{\ref{se:om}a}\label{seom:a}
\\ (\mathcal P_{\CC^{\up}}f)(z)&=\int_{-\infty}^{+\infty}  f(t) \Poi (t,z)\dd t
=\int_{\RR_{\pm\infty}}\Poi (t,z)\,f(t)\dd \lambda_{\RR}(t),\quad z\in \CC^{\up}.
 \tag{\ref{se:om}b}\label{seom:b}
 \end{align}
 \end{subequations}
В случае $D=\CC^{\up}$ в обозначениях левых частей из \eqref{df:kP} и \eqref{se:om} нижний индекс ${\CC^{\up}}$, как правило, опускаем, т.\,е. пишем просто $ \Poi (t,z)$,  $\omega(z, B)$ и $(\mathcal P f)(z)$.
Определение \eqref{seom:a} 
можно распространить на $z=x\in {\RR}$, полагая $\omega(x,B)=\delta_x(B)$ --- мера Дирака в точке $x$.
Для функций $f$ с полунепрерывным сверху сужением $f\bigm|_{\RR}$ и  со значениями в $\RR_{-\infty}$ при условии
\begin{equation}\label{J2+}
	J_2^+[f]:=\int_{-\infty}^{+\infty} \frac{f^+(t)}{1+t^2}\dd t <+\infty
\end{equation}
 будем полагать $(\mathcal P f)(x)=f(x)\in \RR_{-\infty}$ для $x\in \RR$.

Гармоническая мера для верхней полуплоскости  имеет известный простой геометрический смысл. Для каждого отрезка 
$[t_1,t_2]\subset \RR$ --- это угол, под которым виден этот отрезок из точки $z\in \CC^{\up}$, деленный на $\pi$ \cite[гл.~I, 3]{Garnett}. Аналитически это записывается как 
\begin{subequations}\label{multo}
\begin{align}
\omega\bigl(z, [t_1,t_2]\bigr)=
\frac{1}{\pi}(\arctg d_2	-\arctg d_1), 
\tag{\ref{multo}a}\label{mult}\\
d_2:= \frac{t_2-\Re z}{\Im z}, \quad d_1:=\frac{t_1-\Re z}{\Im z}\,.
\tag{\ref{multo}b}\label{d1d2}
\end{align}
\end{subequations}

\subsection{Представления и  оценки гармонической меры}\label{s:dal}
Приведенные ниже утверждения данного подраздела ~\ref{p_3_1} и следующего подраздела ~\ref{s:dal} часто формулируются в более общей форме, чем необходимо для наших целей. Причины этого состоят как в исключительной важности и пользе  понятия гармонической меры и ее количественных характеристик в теориях функций комплексного переменного и потенциала, так и в заготовке материалов для исследований  в направлениях, отличных от затронутых в настоящей работе. Всюду в  доказательствах мы предпочитаем использовать геометрический смысл гармонической меры \eqref{multo}, хотя все утверждения этого родраздела ~\ref{s:dal}
могут быть установлены и через представление  \eqref{seom:a} гармонической меры интегралом Пуассона. Именно последний подход был использован в \cite{Kha89}--\cite[гл.I, II]{KhDD92}. 

\begin{propos}\label{pr:gm} Пусть  $-\infty<t_1<t_2<+\infty$.
Если  точка $z\in \CC^{\up}$ лежит вне замкнутого полукруга из верхней полуплоскости с ди\-а\-м\-е\-т\-р\-ом-от\-р\-е\-з\-к\-ом  $[t_1,t_2]\subset \RR\subset \CC$, что в аналитической форме записывается как
\begin{subequations}\label{krt}
\begin{align}
\Bigl|z-\frac{t_2+t_1}{2}\Bigr|> \frac{t_2-t_1}{2}\,,
\tag{\ref{krt}a}\label{zt1t2}\\
\intertext{или в эквивалентной форме}
	|z|^2-\Re z\, (t_1+t_2)+t_1t_2>0,
\tag{\ref{krt}b}\label{kr}
\end{align}		
\end{subequations}
то имеют место равенство и оценка сверху 
\begin{subequations}\label{eo}
\begin{align} 
\omega\bigl(z,[t_1,t_2]\bigr)&=\frac{1}{\pi}\arctg\frac{(t_2-t_1)\,\Im z}{|z|^2-\Re z \, (t_1+t_2)+t_1t_2}&
\tag{\ref{eo}a}\label{eo:a}
\\
&\leq \frac{1}{\pi}\,\frac{(t_2-t_1)\,\Im z}{|z|^2-\Re z \, (t_1+t_2)+t_1t_2}\, .
 \tag{\ref{eo}b}\label{eo:b}
\end{align}
\end{subequations}

Если  точка $z\in \CC^{\up}$ лежит в открытом полукруге из $\CC^{\up}$ с ди\-а\-м\-е\-т\-р\-ом-от\-р\-е\-з\-к\-ом  $[t_1,t_2]$, что в аналитической форме записывается как
\begin{subequations}\label{krtvn}
	\begin{align}
	\Bigl|z-\frac{t_2+t_1}{2}\Bigr|< \frac{t_2-t_1}{2}\,,
	\tag{\ref{krtvn}a}\label{zt1t2vn}\\
	\intertext{или в эквивалентной форме}
	|z|^2-\Re z\, (t_1+t_2)+t_1t_2<0,
	\tag{\ref{krtvn}b}\label{krvn}
	\end{align}		
\end{subequations}
то имеют место равенство и оценка снизу 
\begin{subequations}\label{eovn}
	\begin{align} 
	\omega\bigl(z,[t_1,t_2]\bigr)&=1+\frac{1}{\pi}\arctg\frac{(t_2-t_1)\,\Im z}{|z|^2-\Re z \, (t_1+t_2)+t_1t_2}&
	\tag{\ref{eovn}a}\label{eovn:a}
	\\
	&\geq 1+\frac{1}{\pi}\,\frac{(t_2-t_1)\,\Im z}{|z|^2-\Re z \, (t_1+t_2)+t_1t_2}\, .
	\tag{\ref{eovn}b}\label{eovn:b}
	\end{align}
\end{subequations}

Наконец, если точка $z\in \CC^{\up}$ лежит на полуокружности   из $\CC^{\up}$  с ди\-а\-м\-е\-т\-р\-ом-от\-р\-е\-з\-к\-ом  $[t_1,t_2]$, что в аналитической форме записывается как
\begin{subequations}\label{krtvn0}
	\begin{align}
	\Bigl|z-\frac{t_2+t_1}{2}\Bigr|= \frac{t_2-t_1}{2}\,,
	\tag{\ref{krtvn0}a}\label{zt1t2vn0}\\
	\intertext{или в эквивалентной форме}
	|z|^2-\Re z\, (t_1+t_2)+t_1t_2=0,
	\tag{\ref{krtvn0}b}\label{krvn0}
	\end{align}		
\end{subequations}
то $\omega\bigl(z,[t_1,t_2]\bigr)=1/2$.
\end{propos}
\begin{proof} 
В случае  \eqref{krt} в обозначениях  \eqref{d1d2}  имеем
\begin{equation*}
d_2d_1=\frac{t_2t_1-\Re z\,(t_2+t_1)+(\Re z)^2}{(\Im z)^2} 
=  \frac{t_2t_1-\Re z\,(t_2+t_1)+ |z|^2}{(\Im z)^2}-1\overset{\eqref{kr}}{>}-1,
\end{equation*}
что позволяет применить формулу разности арктангенсов в \eqref{mult} в форме
\begin{equation*}
\arctg d_2	-\arctg d_1=\arctg\frac{d_2-d_1}{1+d_1d_2}\overset{\eqref{d1d2}}{=}
\arctg\frac{(t_2-t_1)\,\Im z}{|z|^2-\Re z \, (t_1+t_2)+t_1t_2}\,.
\end{equation*}
Правая часть здесь дает  \eqref{eo:a}, а затем и \eqref{eo:b} ввиду возрастания $\arctg$.

В случае  \eqref{krtvn} в обозначениях  \eqref{d1d2}  имеем $d_2d_1<-1$ и $d_2\overset{\eqref{d1d2}}{>}0$,
что позволяет применить формулу разности арктангенсов в \eqref{mult} в форме
\begin{equation*}
\arctg d_2	-\arctg d_1=\pi+\arctg\frac{d_2-d_1}{1+d_1d_2}\overset{\eqref{d1d2}}{=}
\pi+\arctg\frac{(t_2-t_1)\,\Im z}{|z|^2-\Re z \, (t_1+t_2)+t_1t_2}\,.
\end{equation*}
Правая часть здесь дает  \eqref{eovn:a}, а затем и \eqref{eovn:b} ввиду возрастания функции $\arctg$ 
и отрицательности ее значений в правой части \eqref{eovn:a}.

Случай \eqref{krtvn0} по геометрическому смыслу гармонической меры очевиден. 
\end{proof}
\begin{propos}\label{es:qmlw}
	При $b>1$  для точки $z\in \CC^{\up}$,  лежащей  вне открытого  полукруга из  верхней полуплоскости 	с центром 	$(t_2+t_1)/2\in [t_1,t_2]$ радиуса $b(t_2-t_1)/2$, т.\,е.    при 
	условии\/   	{\rm  (ср. с \eqref{zt1t2})}
	\begin{equation}\label{zt1t2b}
	\Bigl|z-\frac{t_2+t_1}{2}\Bigr|\geq b\,\frac{t_2-t_1}{2}\,, \quad z\in \CC^{\up},\quad \text{где \; $b>1$ фиксировано},
	\end{equation}	
	справедлива нижняя оценка
	\begin{subequations}\label{z2lw}
		\begin{align}
		\omega\bigl(z,[t_1,t_2]\bigr)&\geq \frac{b-1}{2\pi b}\,
 \frac{(t_2-t_1)\,\Im z}{|z|^2-\Re z \, (t_1+t_2)+t_1t_2}
\tag{\ref{z2lw}a}\label{z2lwb}\\
&\geq \frac{b-1}{2\pi b}\, \frac{(t_2-t_1)\,\Im z}{\bigl(|z|+t_1\bigr)\bigl(|z|+t_2\bigr)}\,.
\tag{\ref{z2lw}b}\label{z2lwbb}
		\end{align}
	\end{subequations}
\end{propos}
\begin{proof} Очевидная версия равенства \eqref{eo:a} предложения \ref{pr:gm} ---
	\begin{lemma}\label{l:oup} Пусть $r>0$ и $z=te^{i\theta}$ с $t>r$, $0\leq \theta \leq \pi$. Тогда 
	\begin{equation}\label{eo_rz}
	\omega\bigl(z,[-r,r]\bigr)=\frac{1}{\pi}\arctg\frac{2rt\sin \theta}{t^2-r^2},
	\end{equation}
где макcимум при фиксированных $t>r>0$  достигается при $\sin \theta =1$, а  при $t=br$, где $b>1$,  равен
\begin{equation}\label{m:od}
\omega\bigl(ibr,[-r,r]\bigr)=\frac{1}{\pi}\arctg\frac{2b}{b^2-1}\, .
\end{equation} 
\end{lemma}
Заключение леммы \ref{l:oup} о максимуме  с равенством \eqref{m:od} для отрезка $[t_1,t_2]$ 	означает, что 
 при всех $z\in \CC^{\up}$, удовлетворяющих условию \eqref{zt1t2b}, в обозначении 
 \begin{equation}\label{xfr}
 x:=\frac{(t_2-t_1)\,\Im z}{|z|^2-\Re z \, (t_1+t_2)+t_1t_2}
  \end{equation}
  из геометрического смысла гармонической меры и возрастания функции $\arctg$ следует
  $0\leq x\leq \dfrac{2b}{b^2-1}$. Тогда 
 в силу вогнутости функции $\arctg$ на $\RR^+$
 \begin{equation}\label{arcx}
\arctg x\geq \frac{\arctg\frac{2b}{b^2-1}}{\frac{2b}{b^2-1}} \,x 
\quad \text{при} 
\quad 0\leq x\leq \dfrac{2b}{b^2-1}\,,
 \end{equation}
 откуда согласно  неравенствам 
\begin{equation*} 
\frac{1}{b-1}\leq\frac{2b}{b^2-1}\leq \frac{2}{b-1}\quad\text{при $b>1$}
 \end{equation*} 
 следует 
 \begin{equation*}
 \arctg x\overset{\eqref{arcx}}{\geq} 
 \frac{1}{2}\Bigl((b-1)\arctg\frac{1}{b-1}\Bigr)\cdot x\geq \frac12\frac{b-1}{b}\, x
 \quad \text{при} 
 \quad 0\leq x\leq \dfrac{2b}{b^2-1}\,,
 \end{equation*}
 где последнее неравенство основано на неравенстве $\arctg \frac{1}{y}\geq \frac{1}{1+y}$, справедливом при всех $y=b-1> 0$.
 Отсюда ввиду \eqref{xfr} получаем оценку снизу \eqref{z2lwb}. Для перехода 
 к   \eqref{z2lwbb} оценим сверху последний  знаменатель  в \eqref{z2lwb}:
 \begin{equation*}
 0 < |z|^2-\Re z\, (t_1+t_2)+t_1t_2\leq |z|^2+|z|\bigl(|t_1|+|t_2|\bigr)+|t_1||t_2|
 =\bigl(|z|+t_1\bigr)\bigl(|z|+t_2\bigr),
 \end{equation*}
 что и завершает доказательство.
 \end{proof}

Следующее утверждение уточняет \cite[лемма 1.1.1]{KhDD92} и \cite[лемма 1.1]{Kha91}  в части оценок сверху. 
\begin{propos}\label{pr:gmA} Пусть  $-\infty<t_1<t_2<+\infty$, $z\in \CC^{\up}$,  $a\in (0,1)$. Из
\begin{equation}\label{minM}
a|z|\geq \max\bigl\{|t_1|,|t_2|\bigr\}
\end{equation}
следует  оценка 
\begin{equation}\label{afminAM}
\omega\bigl(z,[t_1,t_2]\bigr)\leq \frac{t_2-t_1}{\pi (1-a)^2}\, \Im \frac{1}{\bar z}\,,
\end{equation}
а при условии 
\begin{equation}\label{min}
	 a\cdot \min\limits_{t\in [t_1, t_2]} |t|\geq |z|
 \end{equation}
---  оценка
\begin{equation}\label{afmin}
	\omega\bigl(z,[t_1,t_2]\bigr)\leq \frac{(t_2-t_1)\,a^2}{\pi (1-a)^2}\, \Im \frac{1}{\bar z}\,,
\end{equation}
\end{propos}
\begin{proof}  Левую часть \eqref{kr} можно оценить снизу:
\begin{equation}\label{zsn}
	|z|^2-\Re z\, (t_1+t_2)+t_1t_2\geq |z|^2-|z|\, (t_1+t_2)+t_1t_2=\bigl(t_1-|z|\bigr)\,\bigl(t_2-|z|\bigr).
\end{equation}
При условии \eqref{minM} оба сомножителя в правой части 
\eqref{zsn} отрицательны, т.\,е. выполнено условие \eqref{kr}. Следовательно, 
\begin{equation*}
\omega\bigl(z,[t_1,t_2]\bigr)\overset{\eqref{zsn}, \eqref{eo:b}}{\leq} \frac{1}{\pi}\,\frac{(t_2-t_1)\,\Im z}{\bigl(|z|-t_1\bigr)\,\bigl(|z|-t_2\bigr)}
\overset{\eqref{minM}}{\leq} \frac{1}{\pi}\,\frac{(t_2-t_1)\,\Im z}{\bigl(|z|-a|z|\bigr)\,\bigl(|z|-a|z|\bigr)}\,, 
\end{equation*}
где правая часть совпадает с правой частью в \eqref{afminAM}. 
Условие \eqref{min}, в частности, означает, что точки $t_1,t_2\in \RR$ расположены по одну сторону от нуля и пара сомножителей в правой  части \eqref{zsn} одинакового знака, т.\,е. выполнено условие \eqref{kr}. Следовательно, 
\begin{equation*}
	\omega\bigl(z,[t_1,t_2]\bigr)\overset{\eqref{zsn}}{\leq} \frac{1}{\pi}\,\frac{(t_2-t_1)\,\Im z}{\bigl(t_1-|z|\bigr)\,\bigl(t_2-|z|\bigr)}
	\overset{\eqref{min}}{\leq} \frac{1}{\pi}\,\frac{(t_2-t_1)\,\Im z}{\bigl(|z|/a-|z|\bigr)\,\bigl(|z|/a-|z|\bigr)}\,, 
\end{equation*}
где правая часть совпадает с правой частью в \eqref{afmin}. 
\end{proof}
Следующее утверждение уточняет \cite[лемма 1.1.1]{KhDD92} и  \cite[лемма 1.1(1), (1.2)]{Kha91}  в части оценки снизу.
\begin{propos}\label{pr:gmA+} В обозначениях предложения\/ {\rm \ref{pr:gmA}} 
	при условии\/ \eqref{minM}

имеет место нижняя оценка 
	\begin{equation}\label{afminAM-}
	\omega\bigl(z,[t_1,t_2]\bigr)\geq (t_2-t_1) \frac{1-a}{8\pi}\,
\Im \frac{1}{\bar z}\,.
	\end{equation}
\end{propos}

\begin{proof} При условии \eqref{minM} элементарные геометрические рассуждения показывают, что выполнено  \eqref{zt1t2b} с $b:=1/a$. Кроме того,  при условии  \eqref{minM}, оставаясь в рамках  \eqref{kr}, имеем оценку сверху 
\begin{equation}\label{eo:b-}
0 < \bigl(|z|+t_1\bigr)\bigl(|z|+t_2\bigr)
\overset{\eqref{minM}}{\leq} |z|^2(1+a)^2.  
\end{equation}
Отсюда по предложению \ref{es:qmlw} из оценки \eqref{z2lwbb} с $b=1/a$ получаем 
\begin{equation*}
\omega\bigl(z,[t_1,t_2]\bigr)\geq 
\frac{1/a-1}{2\pi/a}\,\frac{(t_2-t_1)\,\Im z}{|z|^2(1+a)^2}
=\frac{1-a}{2\pi(1+a)^2}\,(t_2-t_1)\,\Im \frac{1}{\bar z}\, ,
\end{equation*}
что  в силу $0< a<1$ дает нижнюю оценку \eqref{afminAM-}.
\end{proof}

Для точек $z\in \CC^{\up}$  их  аргументы $\arg z$ берутся из интервала $(0,\pi)$.
\begin{propos}\label{pr:gmAa} Пусть  $0\leq t_1<t_2<+\infty$, $z\in \CC^{\up}$. Из
\begin{equation}\label{mina}
\pi <\arg z\leq \frac{\pi}{2}, \quad\text{т.\,е. $\cos \arg z\leq 0$,}
 \end{equation}
 следует  оценка сверху
\begin{equation}\label{afmina}
	\omega\bigl(z,[t_1,t_2]\bigr)\leq (t_2-t_1)\, \frac{1}{\pi}\Im \frac{1}{\bar z}\,.
\end{equation}
Из условия 
\begin{equation}\label{minMa}
-1< \cos \arg z<\frac{2\sqrt{t_1t_2}}{t_1+t_2} 
\end{equation}
следует оценка сверху
\begin{equation}\label{afminAMa}
	\omega\bigl(z,[t_1,t_2]\bigr)\leq\frac{( t_2-t_1)\, \Im z}{\pi \bigl(|z|-\sqrt{t_1t_2}\,\bigr)^2}\,.
\end{equation}
Пусть для некоторого $a\in (0,1)\subset\RR^+$ выполнено  условие
\begin{equation}\label{minMaa}
-1< \cos \arg z\leq \frac{2a\sqrt{t_1t_2}}{t_1+t_2} \,.
\end{equation}
Тогда имеет место оценка сверху
\begin{equation}\label{afminAMaa}
	\omega\bigl(z,[t_1,t_2]\bigr)\leq\frac{ t_2-t_1}{\pi \bigl(1-a^2\bigr)}\, \Im \frac{1}{\bar z}\,.
\end{equation}
\end{propos}
\begin{proof} При условии  \eqref{mina} для $[t_1,t_2]\subset \RR^+$ имеем
\begin{equation*}
	|z|^2-\Re z\, (t_1+t_2)+t_1t_2\geq |z|^2>0,
\end{equation*}
т.\,е. выполнено \eqref{kr} и по предложению \ref{pr:gm}  из  \eqref{eo:b} сразу следует \eqref{afmina}.

При условии \eqref{minMa} 
\begin{equation}\label{sqrt}
|z|^2-\Re z\, (t_1+t_2)+t_1t_2> |z|^2-2|z|\sqrt{t_1t_2}+t_1t_2=\bigl(|z|-\sqrt{t_1t_2}\bigr)^2\geq 0,
\end{equation}
т.\,е. выполнено \eqref{kr} и по предложению \ref{pr:gm}  из  \eqref{eo:b} сразу следует \eqref{afminAMa}.
Вариация оценки снизу  \eqref{sqrt} при чуть более жестком условии \eqref{minMaa} дает 
\begin{multline}\label{zsnizu}
	|z|^2-\Re z\, (t_1+t_2)+t_1t_2\geq  |z|^2-2a|z|\sqrt{t_1t_2}+t_1t_2
	\\
	=(1-a^2)|z|^2+\bigl(a|z|-\sqrt{t_1t_2}\bigr)^2\geq (1-a^2)|z|^2>0,
\end{multline}
откуда по предложению \ref{pr:gm}  из  \eqref{eo:b} получаем \eqref{afminAMaa}.
\end{proof}

\begin{propos}\label{r:+} Пусть $0<t_1<t_2 <+\infty$ и для $z\in \CC^{\up}$ выполнено условие 
	\eqref{minMa}. 		Тогда справедлива оценка снизу
	\begin{equation}\label{essnizo}
	\omega\bigl(z,[t_1,t_2]\bigr)\geq \frac{1}{8\pi}\frac{t_1}{t_2}\, (t_2-t_1)\,\Im \frac{1}{\bar z}\, .
	\end{equation}
В частности, если $t_2\leq At_1$ для некоторого числа $A>1$, то 
	\begin{equation}\label{essnizoA}  
	\omega\bigl(z,[t_1,t_2]\bigr)\geq \frac{1}{8\pi A}\, (t_2-t_1)\,\Im \frac{1}{\bar z}\, .
	\end{equation}
\end{propos}	
\begin{proof}	
		 Предположим сначала, что точка $z$ такова, что выполнено условие 
	 \eqref{minMaa} при некотором выбранном $a\in (0,1)$.  Из элементарных геометрических построений при  $0<t_1<t_2$ и условии \eqref{minMaa} нетрудно вывести, что точка $z$ лежит вне полукруга из предложения \ref{es:qmlw},  т.\,е. выполнено условие \eqref{zt1t2b} предложения \ref{es:qmlw} с числом 
	 \begin{equation}\label{dfb}
	 b:=\sqrt{1+\frac{4(1-a^2)t_1t_2}{(t_2-t_1)^2}}>1\, .
	 \end{equation}
	Тогда имеют место оценки снизу 	\eqref{z2lw} с выбранным в  \eqref{dfb} значением $b>1$. В частности, из 
	\eqref{z2lwb} в сочетании с оценкой \eqref{zsnizu}, выполненной при условии \eqref{minMaa}, имеем 
		\begin{equation}\label{z2lwb+-}
	\omega\bigl(z,[t_1,t_2]\bigr)\geq \frac{b-1}{2\pi b}\, \frac{(t_2-t_1)\,\Im z}{(1-a^2)|z|^2}\geq
	 \Bigl(1-\frac{1}{b}\Bigr)\frac{1}{4\pi (1-a)}\, (t_2-t_1)\,\Im \frac{1}{\bar z}	\, .
	\end{equation}
Оценим снизу первый сомножитель в правой части \eqref{z2lwb+-}:
\begin{equation}\label{1fr1b}
1-\frac{1}{b}= 1-\frac{1}{\sqrt{1+\frac{4(1-a^2)t_1t_2}{(t_2-t_1)^2}}}\geq 
1-\frac{1}{\sqrt{1+4(1-a^2){t_1}/{t_2}}}
\end{equation}
Здесь 
$0<x:=4(1-a^2)\,{t_1}/{t_2}<4$ и 
 при $x\in [1,4]$ имеет место элементарное неравенство 
$1-\frac{1}{\sqrt{1+x}}\geq \frac{1}{8}\, x$,
откуда можем продолжить оценку снизу для правой части \eqref{1fr1b} в виде
\begin{equation*}
1-\frac{1}{b}\geq \frac{1}{8} \, 4(1-a^2)\,\frac{t_1}{t_2} \geq \frac{1-a}{2} \frac{t_1}{t_2}\, .
\end{equation*} 
Используя эту оценку в  правой части \eqref{z2lwb+-}, получаем требуемую оценку  \eqref{essnizo} для произвольного $a\in (0,1)$. Поскольку правая часть в  \eqref{essnizo} не зависит от $a$, можем в условии \eqref{minMaa} по непрерывности устремить $a$ к единице. Таким образом, оценка  \eqref{essnizo} выполнена и при условии  	\eqref{minMa}.
\end{proof}

\begin{remark}\label{r:t1t2+-}
	Условие $0\leq t_1<t_2<+\infty$ предложения \ref{pr:gmAa} легко адаптируется на случай $-\infty<t_1<t_2\leq 0$. 
Таким образом, при $a\in (0,1)$ как в случае $0\leq t_1<t_2<+\infty$, так и в случае $-\infty<t_1<t_2\leq 0$  в той или иной мере  неохваченным оценками предложений \ref{pr:gmA} и \ref{pr:gmAa} осталось множество 
\begin{equation}\label{xset}
\left\{z\in \overline{\CC}^{\up}\colon \; at_1 \leq |z|<\frac{1}{a}\,t_2,\;
\cos \arg z\geq \frac{2a\sqrt{t_1t_2}}{t_1+t_2}\, \right\},
\end{equation}
а для  $-\infty<t_1\leq 0\leq t_2<+\infty$, $t_1\neq t_2$, с тем же $a\in (0,1)$ --- круг 
\begin{equation}\label{xsetd}
D\bigl(0,\max\{|t_1|, |t_2|\}/a\bigr).
\end{equation}
\end{remark}

Из элементарной геометрии, к примеру, из теоремы косинусов, легко следует
\begin{propos}\label{prxset} В описанных в замечании\/ {\rm \ref{r:t1t2+-}} случаях множество \eqref{xset} и круг
\eqref{xsetd}	содержатся в круге
\begin{equation}\label{DsD}
	D\bigl((t_1+t_2)/2, (a+1/a)\max\{|t_1|,|t_2|\}\bigr)\supset D\bigl(\max\{|t_1|,|t_2|\}\bigr).
\end{equation}
\end{propos}

Следующее утверждение заполняет отмеченный в замечании \ref{r:t1t2+-}  пробел и дает в определенном смысле точные  оценки через длину (радиус) отрезка $[t_1,t_2]$.
\begin{propos}\label{pr:eshmb} Пусть  $0\leq t_1<t_2<+\infty$, $r:=(t_2-t_1)/2$ --- радиус
	отрезка $[t_1,t_2]\subset \RR^+$ с центром $t_0:=(t_1+t_2)/2$, 
	Тогда для всех $z\in \overline{\CC}^{\up}$ из замкнутой верхней полуплоскости справедливы  оценки сверху
	\begin{equation}\label{es:upom}
	\hspace{-2mm}	
	\omega\bigl(z,[t_1,t_2]\bigr)\leq \begin{cases}
		  1  \quad\text{при всех $z\in {D}(t_0,r)\cup\{t_1\}\cup \{t_2\} $},\\
				  {1}/{2}  \quad\text{при всех $z\in \CC^{\up} \cap \partial {D}(t_0,r)$},\\
		\dfrac{1}{\pi}\, \arctg \dfrac{2r|z-t_0|}{|z-t_0|^2-r^2} 
		\leq 		\dfrac{2r|z-t_0|}{\pi\bigl(|z-t_0|^2-r^2\bigr)}\,,\;z\notin \overline{D}(t_0,r),
\end{cases}
\quad z\in \overline{\CC}^{\up},
	\end{equation}
	которые для каждого $t\in \RR^+$ точны на каждой верхней полуокружности
	\begin{equation}\label{df:asc}
	\partial D^{\up}(t_0,t):=\bigl\{z\in \overline{\CC}^{\up} \colon |z-t_0|=t\bigr\},
	\quad  D^{\up}(t_0,t):={\CC}^{\up}\cap  D(t_0,t).
	\end{equation}
	Кроме того, при всех $z\overset{\eqref{df:asc}}{\in} D^{\up}(t_0,r) $ имеет место оценка снизу
	\begin{equation}\label{essn_o}
		\omega\bigl(z,[t_1,t_2]\bigr)\geq 1-\dfrac{1}{\pi}\, \arctg \dfrac{2r|z-t_0|}{r^2-|z-t_0|^2}
		\geq 1-\dfrac{1}{\pi}\, \dfrac{2r|z-t_0|}{r^2-|z-t_0|^2} \,.
	\end{equation}
	\end{propos}
	\begin{proof}
Верхняя оценка через $1$ очевидна, поскольку гармоническая мера вероятностная, и это значение при $z\in  	\partial D(t_0,t)^{\up}$ с $t<r$  достигается на верхней полуокружности \eqref{df:asc}  при $z=t_0\pm t=t_0\pm|z-t_0|\in [t_1,t_2]$. При этом на верхней 
полуокружности $\CC^{\up} \cap \partial {D}(t_0,r)$ из геометрического смысла гармонической  меры левая часть \eqref{es:upom} тождественно равна $1/2$. 
 При $z\notin  \clos {D}^{\up}(t_0,t) $ с помощью сдвига $z\mapsto z-t_0$ удобно перейти 
 к отрезку $ [-r,r]$, $t_0=0$, $|z-t_0|=t$
 и воспользоваться леммой \ref{l:oup} с $\sin \theta =1$, когда максимум достигается. Это  доказывает как последние оценки из  \eqref{es:upom}, так и ее точность.
 
После сдвига $z\mapsto z-t_0$, $z\in \overline\CC^{\up}$, для доказательства оценки снизу \eqref{essn_o}  достаточна   версия равенства \eqref{eovn:a} предложения \ref{pr:gm} (ср. с леммой  \ref{l:oup}) ---
\begin{lemma}\label{l:oups} Пусть\/ $r>0$ и\/ $z=te^{i\theta}$ с\/ $0<t<r$, $0\leq \theta \leq \pi$. Тогда 
	\begin{equation*}
	\omega\bigl(z,[-r,r]\bigr)=1+\frac{1}{\pi}\arctg\frac{2rt\sin \theta}{t^2-r^2}
	=1-\frac{1}{\pi}\arctg\frac{2rt\sin \theta}{r^2-t^2},
	\end{equation*}
	где минимум при фиксированных $r>t>0$  достигается при $\sin \theta =1$.
\end{lemma}
Это завершает доказательство.
\end{proof}
\subsection{Классическое выметание из полуплоскости и условие Бляшке для полуплоскости}	
 
 Все характеристики  роста (соответственно исчезания) и понятия принадлежности классу (ра-)\-с\-х\-о\-д\-и\-м\-о\-с\-ти около $\infty$ (соответственно нуля) 
из \S~\ref{sss:sh}  для заряда $\nu\in \mathcal M(\CC)$  полностью   определяются одноименными 
характеристиками и понятиями для сужений возрастающей функции $|\nu|^{\rad}$ на какой-либо  луч $[r_0,+\infty)$ (соответственно интервал $(0,r_0]$), где $0<r_0\in \RR^+$. 
При этом фразы {\it <<около\/ $0$>>\/} и  {\it <<около\/ $+\infty$>>\/}
эквивалентны фразам {\it <<около нуля>>\/} и {\it <<около бесконечности>>\/}
или {\it <<около\/ $\infty$>>.\/}

Характеристики роста и исчезания  заряда $\nu\in \mathcal M(\CC)$ дополняет 
\begin{definition}\label{c:akh}
	Заряд $\nu\in  \mathcal M(\CC)$  удовлетворяет {\it условию Бляшке около бесконечности для верхней полуплоскости\/}  $\CC^{\up}$ (см. и ср. с \cite[гл.~II, 2]{Garnett} и с определением функций класса A из \cite[гл.~V]{Levin56}),  если 
\begin{equation}\label{ineq:akh}
\int\limits_{\CC^{\up}\setminus D(r_0)} \Im \frac{1}{\bar z} \dd |\nu|(z)<+\infty \quad\text{для некоторого $r_0\in \RR^+_*$,}
\end{equation}
где подынтегральная функция в  \eqref{ineq:akh}, очевидно, удовлетворяет условиям
\begin{equation}\label{co:imz} 
0\leq  \Im \frac{1}{\bar z}\leq \frac{1}{|z|}\quad\text{для всех  $z\in \CC^{\up}_*$.}
\end{equation}
Заряд $\nu\in  \mathcal M(\CC)$  удовлетворяет {\it условию Бляшке около нуля для верхней полуплоскости\/ $\CC^{\up}$,}  если 
\begin{equation}\label{ineq:akh0}
\int\limits_{\CC^{\up}\cap D_*(r_0)} \Im \frac{1}{\bar z} \dd |\nu|(z)<+\infty \quad\text{для некоторого $r_0\in \RR^+_*$.}
\end{equation}
\end{definition}
Отметим очевидные следствия неравенств \eqref{co:imz}
и предложений \ref{pr:fcc}\eqref{i:i1} и \ref{pr:fcc0}\eqref{i:i10}.
		\begin{propos}\label{pr:2st} Если сужение $\nu\bigm|_{\CC^{\up}}$ заряда  $\nu\in \mathcal M(\CC)$ принадлежит классу сходимости при порядке  $1$ около бесконечности (нуля), то заряд $\nu$ удовлетворяет условию Бляшке для верхней полуплоскости около бесконечности (соответственно нуля).
\end{propos}
\subsection{ Классическое выметание заряда из верхней полуплоскости}\label{sss:bal} 

\begin{definition}\label{df:clbal} Пусть $\nu\in \mathcal M(\CC)$. Заряд $\nu^{\bal}\in \mathcal M(\CC)\in \mathcal M(\CC)$, определяемый  на каждом борелевском подмножестве $B\in  \mathcal B(\CC)$ равенством 
	\begin{equation}\label{df:bal0up}
	\nu^{\bal}(B):=\int_{\CC^{\up}} \omega(z, B\cap \CC^{\up})\dd \nu(z)+
	\nu\bigl(B\cap {\overline{\CC}_{\lw}}\bigr),
	\end{equation} 
	называем {\it выметанием заряда\/ $\nu$
	из верхней полуплоскости\/} $\CC^{\up}$ на замкнутую нижнюю полуплоскость $\overline{\CC}_{\lw}:=\CC\setminus \CC^{\up}$.
\end{definition} 
Отметим, что выметания $\nu^{\bal}$  положительной меры $\nu\in \mathcal M^+(\CC)$ из определения \ref{df:clbal} при условии ее существования, т.\,е. конечности на множествах   $B\in \mathcal B_{\rm b} (S)$, также определяет положительную меру, сосредоточенную на замкнутой нижней полуплоскости $\overline{\CC}_{\lw}$.
\begin{thm}\label{th:cup} Пусть  заряд $\nu\in \mathcal M(\CC)$ 
	удовлетворяет условию Бляшке около бесконечности для $\CC^{\up}$. Тогда определено выметание $\nu^{\bal}$	из $\CC^{\up}$. 
При этом для функции распределения $\bigl|\nu^{\bal}\bigr|^{\RR}$ на $\RR$ из \eqref{nuR},  $|\nu^{\bal}|$ 
для любых 
\begin{subequations}\label{0rt0}
	\begin{align}
 -\infty<t_1<t_2<+\infty, &\quad a\in (0,1)\subset \RR^{+}
	\tag{\ref{0rt0}a}\label{{0rt0}a}
	\\
\text{при}\quad   t_0:=\frac{t_1+t_2}{2}\, , \quad r&:=\frac{t_2-t_1}{2}, \quad T:=\max\bigl\{|t_1|, |t_2|\bigr\}
\tag{\ref{0rt0}b}\label{{0rt0}b}\\
\text{в обозначении}\quad \nu^{\overline{\up}}&:=\nu\bigm|_{\overline{\CC}^{\up}}
\quad\text{ для  сужения заряда $\nu$} 
\tag{\ref{0rt0}c}\label{{0rt0}c}
\end{align}
\end{subequations} 
 на замкнутую верхнюю полуплоскость имеет место оценка
\begin{multline}\label{nu:estT}
\bigl|\nu^{\bal}\bigr|^{\RR}(t_2)-\bigl|\nu^{\bal}\bigr|^{\RR}(t_1)\leq
|\nu^{\overline{\up}}|(t_0,r)+|\nu^{\overline{\up}}|\bigl(\{t_1, t_2\}\bigr)
+\frac{r(a+1/a)}{T}|\nu^{\overline{\up}}|\bigl(t_0,(a+1/a)T\bigr)
\\
+\frac{2r}{\pi (1-a)^2}\int\limits_{|z|\geq T
	}\Bigl|\Im \frac{1}{z}\Bigr|\dd |\nu^{\overline{\up}}|(z)
+\frac{2r}{\pi}\int\limits_r^{(a+1/a)T}\frac{|\nu^{\overline{\up}}|(t_0,t)}{t^2+r^2}\dd t.
\end{multline}
\end{thm}
\begin{proof}
	Не умаляя общности, можно считать, что 
	$\supp \nu \subset \overline{\CC}^{\up}$, т.\,е.  $\nu^{\overline{\up}}\overset{\eqref{{0rt0}c}}{=}\nu$.
Доказательство оценки \eqref{nu:estT} и будет означать, что определено выметание $\nu^{\bal}$ из $\CC^{\up}$.
	Объединения множества \eqref{xset} с кругом \eqref{xsetd}
	из замечания \ref{r:t1t2+-}   по предложению \ref{prxset}  содержится  в большем круге из  \eqref{DsD}, совпадающем в обозначениях \eqref{{0rt0}b}  с кругом $D\bigl(t_0, (a+1/a)T\bigr)$. Следовательно,  
	по предложениям \ref{pr:gmA} с оценкой \eqref{afminAM} и \ref{pr:gmAa} с оценкой \eqref{afminAMaa}
	получаем 
	\begin{equation}\label{afminAMg}
	\omega\bigl(z,[t_1,t_2]\bigr)\overset{\eqref{afminAM}, \eqref{afminAMaa}}{\leq} \frac{t_2-t_1}{\pi (1-a)^2}\, \Im \frac{1}{\bar z}\overset{\eqref{{0rt0}b}}{=}	\frac{2r}{\pi (1-a)^2}\, \Im \frac{1}{\bar z}
	\end{equation}
для всех точек  $z\notin D\bigl(t_0, (a+1/a)T\bigr)$. Пусть 
\begin{equation}\label{nuiCD}
\nu_{\infty}:=\nu\bigm|_{\CC\setminus D\bigl(t_0, (a+1/a)T\bigr)}
\end{equation}
--- сужение заряда $\nu$ на внешность круга $D\bigl(t_0, (a+1/a)T\bigr)$, включающего в себя по предложению 
\ref{prxset}  меньший круг $D(T)$ из \eqref{DsD}.
Отсюда, интегрируя  обе части неравенства \eqref{afminAMg} по полной вариации $|\nu_{\infty}|\leq |\nu|$, сосредоточенной по определению \eqref{nuiCD} вне круга $D(T)$, имеем
\begin{equation}\label{nuinfty}
\bigl|(\nu_{\infty})^{\bal}\bigr|\bigl([t_1,t_2]\bigr)\overset{\eqref{afminAMg}}{\leq}
\frac{2r}{\pi (1-a)^2}\int\limits_{\CC\setminus D(T)}\Bigl|\Im \frac{1}{z}\Bigr|\dd |\nu|(z).
\end{equation}
Для остальной части $\nu_0\overset{\eqref{nuiCD}}{:=}\nu-\nu_{\infty}$ заряда $\nu$ с полной вариацией 
$|\nu_0|\leq |\nu|$ из оценки \eqref{es:upom} предложения \ref{pr:eshmb} и определения \ref{df:clbal} следует
\begin{multline}\label{nu0btt}
\bigl|(\nu_{0})^{\bal}\bigr|\bigl([t_1,t_2]\bigr)
\overset{ \eqref{df:bal0up}}{\leq} 
\int\limits_{D\bigl(t_0,(a+1/a)T\bigr)}\omega \bigl(z,[t_1,t_2]\bigr) \dd|\nu|(z)
\\
\overset{\eqref{es:upom}}{\leq} |\nu|\bigl(D(t_0,r)\bigr)+|\nu|\bigl(\{t_1,t_2\}\bigr)+
\frac{1}{2}|\nu|\bigl(\CC^{\up} \cap \partial {D}(t_0,r)\bigr)\\
+\int\limits_{D(t_0, (a+1/a)T)\setminus D(t_0,r)} \dfrac{1}{\pi} \arctg \dfrac{2r|z-t_0|}{|z-t_0|^2-r^2}\dd|\nu|(z)
\end{multline}
Последний интеграл здесь равен 
\begin{multline*}
\int\limits_{r}^{(a+1/a)T} \dfrac{1}{\pi} \arctg \dfrac{2rt}{t^2-r^2}\dd |\nu|(t_0,t)=\Bigl|\text{интегрирование по частям}\Bigr|
\\=\frac{1}{\pi}\Bigl(\arctg \frac{2r(a+1/a)T}{(a+1/a)^2T^2-r^2}\Bigr)|\nu|\bigl(t_0,(a+1/a)T\bigr)
-\frac{1}{2}|\nu|(t_0,r)+\frac{2r}{\pi}\int\limits_{r}^{(a+1/a)T}
\frac{|\nu|(t_0,t)}{t^2+r^2}\dd t\\
\leq
\frac{r(a+1/a)}{T}|\nu|\bigl(t_0,(a+1/a)T\bigr)
-\frac{1}{2}|\nu|\bigl(\CC^{\up} \cap \partial {D}(t_0,r)\bigr)+
\frac{2r}{\pi}\int\limits_{r}^{(a+1/a)T} \frac{|\nu|(t_0,t)}{t^2+r^2}\dd t
\end{multline*}
Отсюда, продолжая оценку \eqref{nu0btt} и учитывая \eqref{nuinfty}, ввиду  равенства $\nu=\nu_0+\nu_{\infty}$
и неравенства $|\nu|\leq |\nu_0|+|\nu_{\infty}|$  получаем требуемую в \eqref{nu:estT} оценку.
\end{proof}
\begin{corollary}\label{th:b0+} Пусть $\nu \in \mathcal{M}^+(\CC)$. Следующие три утверждения попарно эквивалентны:
\begin{enumerate}[{\rm (b1)}]
\item\label{b1} существует выметание  $\nu^{\bal}\in \mathcal{M}^+(\CC)$ 	из $\CC^{\up}$,
\item\label{b2} для некоторого подмножества $B\in  \mathcal B_{\rm b} (\RR)$ с $\lambda_{\RR}(B)>0$
\begin{equation}\label{e:lot}
\int_{\CC^{\up}} \omega (z, B )\dd \nu(z)<+\infty ,
\end{equation}
\item\label{b3} мера  $\nu$  удовлетворяет условию Бляшке \eqref{ineq:akh} около $\infty$ для $\CC^{\up}$.
\end{enumerate}
\end{corollary}
\begin{proof} Импликация (b\ref{b3})$\Rightarrow$(b\ref{b1}) --- частный случай теоремы  \ref{th:cup}.
Импликация (b\ref{b1})$\Rightarrow$(b\ref{b2})  --- очевидное следствие равенства \eqref{df:bal0up} определения \ref{df:clbal}. Из  (b\ref{b2}) следует существование чисел $-\infty<t_1<t_2<+\infty$, 
для которых  
\begin{equation}\label{e:loB}
\int_{\CC^{\up}} \omega\bigl(z, (t_1,t_2)\bigr)\dd \nu(z)\overset{\eqref{e:lot}}{<}+\infty
\end{equation}
Положим $a=1/2$ и $T:=\max\bigl\{|t_1|,|t_2|\bigr\}$. Тогда из  оценки снизу \eqref{afminAM-}
 предложения \ref{pr:gmA+} следует
	\begin{equation*}
	(t_2-t_1) \frac{1-1/2}{8\pi}\int_{\CC^{\up}\setminus D(2T)}
\Im \frac{1}{\bar z}\dd \nu(z)\leq \int_{\CC^{\up}} \omega\bigl(z, (t_1,t_2)\bigr)\dd \nu(z)\overset{\eqref{e:loB}}{<}+\infty,
	\end{equation*}
т.\,е. выполнено \eqref{ineq:akh} при $r_0=2T$, и из  (b\ref{b2}) выведено  (b\ref{b3}).
\end{proof}
\begin{remark}
Условие $\lambda_{\RR}(B)>0$ на $B\in  \mathcal B_{\rm b} (\RR)$ можно ослабить. Достаточно предполагать, что (b\ref{b2})  с \eqref{e:lot} выполнено  для некоторого множества $B$ ненулевой логарифмической емкости, т.\,е. неполярного \cite{Rans}.  
\end{remark}
\section{Выметание  заряда на систему лучей}\label{bal_S_0}  
\setcounter{equation}{0}
Значительная часть содержания  этого раздела~\ref{bal_S_0} в некоторой степени развивает и обобщает результаты из \cite{Kha91} и \cite[гл.~I]{KhDD92}.

В разделе~\ref{bal_S_0} всюду  $S$ --- множество лучей на $\CC$ с началом в нуле; $S_*:=S\setminus \{0\}$.  Далее такое множество лучей  $S$ называем {\it системой лучей. \/}
\subsection{ Система лучей}\label{Sray}
    Систему лучей $S$ рассматриваем одновременно и как {\it точечное подмножество в\/} $\CC$, т.\,е. как множество всех точек на лучах, образующих систему $S$.  В частности,  система лучей $S$  {\it замкнутая,\/}
если $S$ как точечное подмножество  замкнуто в $\CC$. Далее всюду, часто не оговаривая специально, {\it рассматриваем только непустые замкнутые системы лучей.\/} Связные компоненты дополнения $\CC\setminus S$ замкнутой системы лучей $S$ называем {\it углами, дополнительными к $S$.\/} И наоборот, точечное множество, образующее конус в $\CC$ над $\RR^+$, рассматриваем и как систему лучей. Так, $S=\RR=\RR^+\cup (-\RR^+)$ --- замкнутая система из пары лучей: 
положительной и отрицательной вещественных полуосей в $\CC$.

Для произвольной точки $z\in \CC_*$, в отличие от договоренности перед предложением \ref{pr:gmAa}, значения ее аргументов $\Arg z\in\RR$   однозначны лишь с точностью до слагаемого, кратного $2\pi$. При этом  $\Arg 0:=\RR$. Произвольному подмножеству $s\subset \RR$ сопоставляем   {\it систему лучей}
\begin{equation}\label{angs}
\angle\, s:=\{z\in \CC \colon s\cap  \Arg z\neq \varnothing\}.
\end{equation}
В таком контексте само подмножество $s\subset \RR$ называем {\it множеством направлений.\/}  И наоборот,  системе лучей $S\subset \CC$ соответствует $2\pi$-периодическое множество $s=\{\theta \in \RR \colon e^{i\theta}\in S\}$  ---  множество направлений, порождающее $S=\angle\, s$. 
В частности, для  интервала $I\subset \RR$  множество $\angle \, I\subset \CC$ --- угол с вершиной  в нуле. Для пары  чисел $\alpha, \beta\in \RR$, в соответствии с \eqref{angs}, 
\begin{subequations}\label{ang}
	\begin{align}
\angle\, [\alpha,\beta]&:=\bigl\{z\in \CC\colon [\alpha, \beta] \cap  \Arg z\neq \varnothing\bigr\} ,
\tag{\ref{ang}a}\label{ang:a}
\\
\angle_* (\alpha,\beta)&{:=}\{0\neq z\in \CC\colon (\alpha, \beta) \cap  \Arg z\neq \varnothing\}
\tag{\ref{ang}b}\label{ang:b}
	\end{align}
\end{subequations}
--- соответственно {\it замкнутый угол\/} (с вершиной в нуле при\/ $\alpha \leq \beta$)  и {\it открытый угол\/} --- оба {\it раствора\/} $\beta-\alpha\in \RR^+$ при $\alpha< \beta\leq \alpha +2\pi$.  Последнее естественно предполагать в записи углов $\angle_* (\alpha,\beta)$, дополнительных к системе лучей $S$.

\subsection{ Гармоническая мера для дополнения системы лучей}\label{df:harmm} 

{\it Редукцией угла\/} $\angle_*(\alpha, \beta)$ ненулевого раствора  {\it к верхней полуплоскости\/} называем конформную замену переменной $z':=(ze^{-i\alpha})^{\frac{\pi}{\beta-\alpha}}$, где $z\in \angle\,[\alpha, \beta]$ 
 и рассматривается ветвь функции $z\mapsto z^{\frac{\pi}{\beta-\alpha}}$, положительная на положительной полуоси $\RR^+$. Пусть $z\in \angle_*(\alpha, \beta)$ и $B'$ --- образ пересечения $B\cap \partial \angle\,[\alpha, \beta]\in \mathcal B(\CC)$  
 при редукции угла $\angle_*(\alpha, \beta)$ к верхней полуплоскости. Тогда, в силу конформной инвариантности гармонической меры, для угла $\angle_*(\alpha, \beta)$ она задается как
 \begin{equation}\label{ioang}
 \omega\bigl(z,B;\angle_*(\alpha, \beta)\bigr)\overset{\eqref{df:o}}{=}\omega(z',B')\overset{\eqref{seom:a}}{=}
 \int_{B'} \frac{1}{\pi}\Im \frac{1}{t-z} \dd \lambda_{\RR}(t).  
 \end{equation}
 
 \begin{propos}\label{aalb} Пусть $a\in (0,1)$. При условии $a|z|\geq r$  
	\begin{equation}\label{aalbaa}
	\omega\bigl(z,\overline{D}(r);\angle_*(\alpha, \beta)\bigr)\leq 
	\frac{2r^{\frac{\pi}{\beta-\alpha}}}{\pi (1-a^{\frac{\pi}{\beta-\alpha}})^2}
	\left(-\Im \frac{1}{(ze^{-i\alpha})^{\frac{\pi}{\beta-\alpha}}}\right),
	\end{equation}
а при условии $ar\geq |z|$ имеет место оценка 
\begin{equation}\label{aalbaad}
\omega\bigl(z,\CC\setminus D(r);\angle_*(\alpha, \beta)\bigr)\leq \frac{2r^{-\frac{\pi}{\beta-\alpha}}}{\pi \bigl(1-a^{\frac{\pi}{\beta-\alpha}}\bigr)^2} \Im\, (ze^{-i\alpha})^{\frac{\pi}{\beta-\alpha}}.
\end{equation}
\end{propos}
\begin{proof} Из  \eqref{ioang} для круга $B=\overline{D}(r)$ имеем 
	\begin{equation*}
	\omega\bigl(z,\overline{D}(r);\angle_*(\alpha, \beta)\bigr)=
	\omega\bigl((ze^{-i\alpha})^{\frac{\pi}{\beta-\alpha}}, [-r^{\frac{\pi}{\beta-\alpha}},r^{\frac{\pi}{\beta-\alpha}}] \bigr),
	\end{equation*}
	Отсюда по предложению \ref{pr:gmA} при $r:=t_2=-t_1>0$ и условии $a|z|\geq r$, соответствующем 
	условию \eqref{minM}, из оценки \eqref{afminAM} после редукции угла $\angle_*(\alpha, \beta)$ к верхней полуплоскости сразу следует \eqref{aalbaa}. Инверсия $\star \colon z\mapsto z^{\star}:= 1/\bar{z}$ по всем  $z\in \angle\,[\alpha, \beta]$ относительно единичной окружности сохраняет, как конформное отображение,  рассматриваемую гармоническую меру и дает оценку \eqref{aalbaad}.
\end{proof}

\begin{definition}[{\cite[определение 1.2]{Kha91}, \cite[определение 1.1.1]{KhDD92}}]\label{seom:aS} Пусть $S$ --- замкнутая система лучей. {\it Гармонической мерой множества\/ $B\in \mathcal B(\CC)$ 
 		в точке\/ $z\in \CC$ для множества\/}-дополнения $\CC\setminus S$ называем функцию
 	\begin{equation}\label{seom:aSf}
 	(z,B)\mapsto \omega (z,B; \CC\setminus S)\in [0,1], \quad z\in \CC, \quad B\in  \mathcal B(\CC),
 	\end{equation}
 	равную гармонической мере $\omega\bigl(z,B;\angle_*(\alpha, \beta)\bigr)$	в точках $z$, попавших в дополнительные к $S$ углы $\angle_*(\alpha, \beta)$, и равную мере Дирака  $\delta_z(B)$ от $B$ при $z\in S$. 
 \end{definition}
 Определение \ref{seom:aS} и обозначение \eqref{seom:aSf} согласованы с определением гармонической меры для
 $\CC^{\up}$ из подраздела ~\ref{p_3_1} и обозначением \eqref{seom:a}, где замкнутая система лучей $S$ образована всеми лучами с вершиной в нуле из $\overline{\CC}_{\lw}$ и $\CC\setminus S=\CC\setminus \overline{\CC}_{\lw}=\CC^{\up}$.

 \subsection{ Выметание заряда на систему лучей и условие Бляшке} 
Редукция угла $\angle_*(\alpha, \beta)$  к    $\CC^{\up}$ из подраздела \ref{df:harmm}   позволяет  обобщить  определение \ref{df:clbal} как 
\begin{definition}\label{df:clbala} Пусть $<-\infty<\alpha <\beta \leq\alpha+ 2\pi<+\infty$
и  $\nu\in \mathcal M(\CC)$. Заряд $\nu^{\bal}_{\CC\setminus \angle_*(\alpha, \beta)}\in \mathcal M(\CC)$, определяемый  на каждом множестве  $B\in  \mathcal B(\CC)$ равенством 
	\begin{equation}\label{df:bal0upa}
	\nu^{\bal}_{\CC\setminus \angle_*(\alpha, \beta)}(B)\overset{\eqref{ioang}}{:=}\int_{\angle_*(\alpha, \beta)} 
 \omega\bigl(z,B;\angle_*(\alpha, \beta)\bigr)\dd \nu(z)+
	\nu\bigl(B\cap {(\CC\setminus \angle_*(\alpha, \beta))}\bigr),
	\end{equation} 
	называем {\it выметанием заряда\/ $\nu$
	из угла\/} $\angle_*(\alpha, \beta)$ на дополнение 
 $\CC\setminus \angle_*(\alpha, \beta)$.
\end{definition}

\begin{definition}[{\cite[определение 1.3]{Kha91}, \cite[определение 1.1.2]{KhDD92}}]\label{df:clbalS} Пусть $\nu\in \mathcal M(\CC)$, $S$ --- замкнутая система лучей на $\CC$. 
	Заряд $\nu_S^{\bal}\in \mathcal M(\CC)$, определяемый  на каждом борелевском множестве $B\in  \mathcal B(\CC)$ равенством 
\begin{equation}\label{df:bal0upS}
\nu^{\bal}_S(B):=\int_{\CC} \omega(z, B; \CC\setminus S)\dd \nu(z),
\end{equation} 
называем {\it выметанием заряда\/ $\nu$ на\/} $S$, или {\it   из\/} $\CC\setminus S$, где нижний индекс $S$ будем опускать, если по (кон)тексту ясно, какая  система лучей $S$ рассматривается.
\end{definition}
 Определения \ref{df:clbal}, \ref{df:clbala} и \ref{df:clbalS} согласованы. Так,  выметание из верхней полуплоскости $\CC^{\up}$ на замкнутую нижнюю полуплоскость $\overline{\CC}_{\lw}$ --- это выметание на замкнутую систему лучей $S$, образованную всеми лучами с вершиной в нуле из $\overline{\CC}_{\lw}$.
Поскольку множество дополнительных к $S$ углов не более чем счетно, 
из определений \ref{df:clbal}--\ref{df:clbalS} сразу следуют свойства
\begin{enumerate}[{\rm (s1)}]
\item\label{s1}   {\it если для зарядов $\nu_1,\nu_2\in \mathcal M(\CC)$ 	существуют выметания $\nu_1^{\bal}, \nu_2^{\bal}\in \mathcal M(\CC)$ на одну и ту же замкнутую систему лучей, то для любых $c_1,c_2\in \RR$ существует  выметание $(c_1\nu_1+c_2\nu_2)^{\bal}=c_1\nu_1^{\bal}+c_2\nu_2^{\bal}$
на ту же систему лучей}  (линейность);
\item\label{s2} {\it для заряда $\nu\in \mathcal M(\CC)$ существание выметания $\nu^{\bal}_{\CC\setminus \angle_*(\alpha,\beta)}\overset{\eqref{df:bal0upa}}{\in} \mathcal M(\CC)$ из каждого угла $\angle_*(\alpha,\beta)$, дополнительного к $S$, эквивалентно существованию   выметания $\nu^{\bal}_S\in \mathcal M (\CC)$, получаемого  применением не более чем счетного числа  выметаний из всех дополнительных  к $S$ углов в смысле определения\/} \ref{df:clbala};
\item\label{s3} {\it выметание $\nu^{\bal}_S$ заряда $\nu \in \mathcal M^+(\CC)$ на систему лучей $S$ при его существовании  --- заряд из  $\mathcal M^+(\CC )$ с носителем $\supp \nu_S^{\bal}\subset 
S$;}
\item\label{s4} {\it выметание $\nu^{\bal}_S$ меры $\nu \in \mathcal M^+(\CC)$ на систему лучей $S$ при его существовании  --- положительная мера из  $\mathcal M^+(\CC )$.} 
\end{enumerate}

Существование выметания   $\nu^{\bal}\in \mathcal M(\CC)$ эквивалентна конечности  \eqref{df:bal0upS} для любого $B\in \mathcal B_{\rm b}(\CC)$. Для существования выметания $\nu^{\bal}\in \mathcal M(\CC)$ достаточно существования выметания $|\nu|^{\bal}\in \mathcal M^+(\CC)$ полной вариации $|\nu|\in \mathcal M^+(\CC)$. По построению для вариаций заряда и его выметания имеют место неравенства
\begin{equation}\label{pv+-}
|\nu^{\bal}|\leq |\nu|^{\bal}, \quad (\nu^{\bal})^{\pm}\leq (\nu^{\pm})^{\bal}\,.
\end{equation}
\begin{example} Неравенства в \eqref{pv+-} могут быть строгими. Пусть $S=\RR$. Рассмотрим заряд $\nu:=\delta_i-\delta_{-i}$ --- разность мер Дирака в точках $i\in \CC^{\up}$ и $-i\in \CC_{\lw}$. Тогда  по определению \ref{df:clbalS} получаем
$\nu_{\RR}^{\bal}=0=|\nu_{\RR}^{\bal}|=(\nu_{\RR}^{\bal})^{\pm}$,
в то время как в обозначении \eqref{nuR} для плотностей  функций  распределения на $\RR$ имеем
 \begin{equation*}
\dd\,\bigl(|\nu|^{\bal}_{\RR}\bigr)^{\RR}(t)=\frac{2}{\pi (1+t^2)}\dd t \, , \quad 
\dd\,\bigl((\nu^{\pm}_{\RR})^{\bal}\bigr)^{\RR}(t)=\frac{1}{\pi (1+t^2)}\dd t ,
\end{equation*}
откуда  $|\nu|^{\bal}_{\RR}(\RR)=2>0=|\nu_{\RR}^{\bal}|(\RR)$ и   $(\nu^{\pm})^{\bal}_{\RR}(\RR)=1>0=(\nu_{\RR}^{\bal})^{\pm}(\RR)$.
\end{example}

Редукция угла к верхней полуплоскости позволяет переформулировать определение \ref{c:akh} для угла как
\begin{definition}\label{c:akhan}
Пусть $<-\infty<\alpha <\beta \leq\alpha+ 2\pi<+\infty$. 	Заряд $\nu\in  \mathcal M(\CC)$  удовлетворяет {\it условию Бляшке около бесконечности для угла  $\angle_*(\alpha,\beta)$,\/}  если 
\begin{equation}\label{ineq:akhan}
\int\limits_{\angle_*(\alpha,\beta)\setminus D(r_0)}-\Im \frac{1}{(ze^{-i\alpha})^{\frac{\pi}{\beta-\alpha}}}\dd |\nu|(z)<+\infty\quad\text{для некоторого $r_0\in \RR^+_*$.}
\end{equation}
где подынтегральная функция в  \eqref{ineq:akh} удовлетворяет условиям
\begin{equation}\label{co:imzan} 
0\leq  -\Im \frac{1}{(ze^{-i\alpha})^{\frac{\pi}{\beta-\alpha}}}\leq \frac{1}{\;|z|^{\frac{\pi}{\beta-\alpha}}}\quad\text{для всех  $z\in \CC^{\up}_*$.}
\end{equation}
Заряд $\nu\in  \mathcal M(\CC)$  удовлетворяет {\it условию Бляшке около бесконечности вне системы  
лучей $S$,\/}  если заряд $\nu$ удовлетворяет условию условию Бляшке около бесконечности для любого 
угла, дополнительного к $S$.
\end{definition}

\begin{thm}\label{th:Sb} Если для заряда  $\nu\in \mathcal M (\CC)$ выполнено условие   Бляшке вне системы лучей $S$, то существует выметание $\nu_S^{\bal}\in \mathcal M (\CC)$. 

Для меры $\nu\in \mathcal M^+(\CC)$ выполнение условия Бляшке вне системы лучей $S$ эквивалентно существованию выметания $\nu_S^{\bal}\in \mathcal M^+(\CC)$.
\end{thm}
\begin{proof} По свойству (s\ref{s2})  импликация  для заряда $\nu\in \mathcal M (\CC)$ --- следствие теоремы \ref{th:cup}, а эквивалентность для меры $\nu\in \mathcal M(\CC)$ вытекает из эквивалентности 
(b\ref{b3})$\Leftrightarrow$(b\ref{b1}) следствия  \ref{th:b0+}.
\end{proof}

Оценку роста выметаемого  заряда или меры дает следующая
\begin{thm}\label{th:0nub} 
	Пусть $\nu\in \mathcal{M}(\CC)$, $S$ --- замкнутая система лучей в $\CC$,  $R\subset \RR_*^+$ --- неограниченное подмножество в $\RR^+$ и существует функция 	
\begin{subequations}\label{rgrB}
	\begin{align}
	g&\colon R\to \RR_*^+,\quad\text{для которой}\quad		\sup_{r\in R} \frac{r}{g(r)} :=a<1, 
	\tag{\ref{rgrB}a}\label{{rgrB}a}\\
C^+_S(r, g;\nu)&:=	\sum_{\angle_*(\alpha,\beta)} r^{\frac{\pi}{\beta-\alpha}}\int\limits_{\angle_*(\alpha,\beta)\setminus D(g(r))}-\Im \frac{1}{(ze^{-i\alpha})^{\frac{\pi}{\beta-\alpha}}}\dd |\nu|(z)<+\infty
	\tag{\ref{rgrB}b}\label{{rgrB}b}
\end{align}
	\end{subequations}
для всех $r\in R$,	где сумма в правой части берется по всем дополнительным к $S$ углам $\angle_*(\alpha,\beta)$. Тогда существует выметание \eqref{df:bal0upS} с оценками 
	\begin{equation}\label{nubrB}
	|\nu^{\bal}|^{\rad} (r)\leq |\nu|\bigl(g(r)\bigr)+\frac{2}{\pi \bigl(1-\sqrt{a}\,\bigr)^2}C^+_S(r,g;\nu)
\quad\text{для всех $r\in R$}.
	\end{equation}   
\end{thm}		
\begin{proof} Достаточно при произвольном $r\in R$ доказать \eqref{nubrB}. Пусть $\nu_{\infty}:=|\nu|\bigm|_{\CC\setminus D(g(r))}\in \mathcal M^+(\CC)$ --- сужение полной вариации $|\nu|$ на внешность круга  	$D(g(r))$, $\nu_r:=|\nu|-\nu_{\infty}\in \mathcal M^+(\CC)$ --- сужение полной вариации $|\nu|$ на  круг  $D(g(r))$. 
Из определяющих соотношений \eqref{df:bal0upS}  предложения \ref{df:clbalS} имеем 
\begin{equation}\label{df:bal0upSg}
(\nu_r)^{\bal}_S\bigl(D(r)\bigr)=\int\limits_{D(g(r))} \omega\bigl(z, D(r); \CC\setminus S\bigr)\dd \nu_r(z)\leq
1\cdot \nu_r^{\rad}(g(r)).
\end{equation}
В силу \eqref{{rgrB}a} при $z\in \supp \nu_{\infty}$ выполнено условие $a|z|\geq r$ 
предложения \ref{aalb} и из оценки \eqref{aalbaa}  в обозначении и при условии \eqref{{rgrB}b} имеем оценку
\begin{equation*}
(\nu_{\infty})^{\bal}_S\bigl(D(r)\bigr)\leq \left( \sup_{0<\beta-\alpha\leq 2\pi} \frac{2}{\pi \bigl(1-a^{\frac{\pi}{\beta-\alpha}}\bigr)^2}\right)\, C^+_S(r,g;\nu)=\frac{2}{\pi \bigl(1-\sqrt{a}\,\bigr)^2} C^+_S(r,g;\nu)
\end{equation*}
через  конечную величину $C^+_S(r,g;\nu)$. Последнее вместе с \eqref{df:bal0upSg}  дает \eqref{nubrB}.

\end{proof}	
Для положительных мер конечного порядка имеет место завершенная
	\begin{thm}\label{th:3p} Пусть $\nu\geq 0$  --- мера конечного типа при порядке $p\in \RR^+$, т.\,е.  $\type_p^{\infty}[\nu]
		\overset{\eqref{senu0:c}}{:=}\type_p^{\infty}\bigl[\nu^{\rad}\bigr]<+\infty$. 
		Эквивалентны  четыре утверждения:
	\begin{enumerate}[{\rm (i)}]	
\item\label{bi} $\int\limits_{\CC} \omega\bigl(z, D(r_0); \CC\setminus S\bigr)\dd \nu(z)\overset{\eqref{df:bal0upS}}{<}+\infty$
для некоторого числа $r_0>0$;
\item\label{bii} для меры $\nu$ в каждом из  дополнительных к $S$ углах $\angle_*(\alpha,\beta)$ раствора $\beta-\alpha\geq \pi/p$, число которых не больше $2p$,  выполнено  условие Бляшке 
\eqref{ineq:akhan} определения\/ {\rm \ref{c:akhan}} около бесконечности;
\item\label{biii} заряд $\nu\in  \mathcal M(\CC)$  удовлетворяет  условию Бляшке 
около бесконечности вне системы лучей $S$;
\item\label{biv} определено выметание $\nu_S^{\bal}\overset{\eqref{df:bal0upS}}{\in}\mathcal{M}(\CC)$ и  $\type_p^{\infty}[\nu_S^{\bal}]<+\infty$.
		\end{enumerate}
	\end{thm}
	\begin{proof} Эквивалентности \eqref{bi}\,$\Leftrightarrow$\,\eqref{bii}\,$\Leftrightarrow$\,\eqref{biv}\,  доказаны в 
\cite[теорема 1.1]{Kha91}, \cite[теорема 1.1.1]{KhDD92} на основе ослабленных версий 
\cite[лемма 1.1]{Kha91}, \cite[лемма 1.1.1]{KhDD92} оценок сверху и снизу гармонической меры из предложений  \ref{pr:gmA} и \ref{pr:gmA+}. Импликация   \eqref{bii}\,$\Rightarrow$\,\eqref{biv} может быть получена  и из теоремы \ref{th:0nub}. Импликация \eqref{biii}\,$\Rightarrow$\,\eqref{bii} очевидна. Рассуждения, доказывающие импликацию  \eqref{bii}\,$\Rightarrow$\,\eqref{biii}, также 
содержаться в \cite[доказательство теоремы 1.1]{Kha91}, \cite[доказательство теоремы 1.1.1]{KhDD92} в части, касающейся выметания меры $\mu_2$ из всех дополнительных к $S$ углов раствора $<\pi/p$.   
	\end{proof}
	
\begin{definition}[{\cite[определение 1.4]{Kha91}, \cite[определение 1.1.3]{KhDD92}}]\label{df:baldop}  Систему лучей $S$  {\it допустимая для меры\/} $\nu\in \mathcal{M}^+(\CC)$ с $\type_p^{\infty}[\nu]<+\infty$, если выполнено одно (любое) из четырех эквивалентных высказываний \eqref{bi}--\eqref{biv} теоремы \ref{th:3p}. 
Система лучей $S$ является {\it $p$-допустимой,\/}
	 если раствор любого дополнительного к $S$ угла менее  $\pi/p$, т.\,е.  $S$ допустимая  для всех
	 $\nu\in \mathcal{M}(\CC)$ с $\type_p^{\infty}[\nu]<+\infty$. 
	\end{definition}

\subsection{ Условие Линделёфа}\label{Lc} 
\begin{definition}[{\cite[следствие 2.1]{Kha91}, \cite[следствие 1.3.1]{KhDD92}}]
	Заряд $\nu\in \mathcal M(\CC)$ удовлетворяет {\it условию Линделёфа при порядке $p\in \NN$,} если 
для некоторого числа $r_0>0$ выполнено асимптотическое соотношение
	\begin{equation}\label{con:Lp}
	\Biggl|\;\int\limits_{D(r)\setminus D(r_0)} \frac{1}{z^p} \dd \nu(z) \Biggr|=O(1) \quad\text{при $r\to +\infty$.}
	\end{equation}
\end{definition}
\begin{thm}[{\cite[следствие 2.1, теорема 3.1]{Kha91}, \cite[следствие 1.3.1, теорема 1.4.1]{KhDD92}}]\label{th:Lc} Пусть  система лучей $S$ допустима для положительной меры $\nu\in \mathcal{M}^+(\CC)$ с $\type_p^{\infty}[\nu]<+\infty$ при некотором $p\in \NN$. Если мера   $\nu$  удовлетворяет условию Линделёфа  \eqref{con:Lp} при порядке $p$, то выметание $\nu_S^{\bal}$ меры $\nu$ на систему лучей $S$ также удовлетворяет условию Линделёфа  при порядке $p$.
\end{thm}	
\section{ Интеграл семейства мер по мере и выметание субгармонических функций}\label{imu} 
\setcounter{equation}{0}

Два различных взгляда на теорию меры и соответственно интегрирования (с одной стороны мера как функция множеств и интеграл Лебега, с другой --- мера как линейный функционал и интеграл Радона) излагаются почти всегда раздельно (ср. \cite{B} и \cite{DSh}; исключение в \cite{M}).
В вопросах, рассматриваемых в настоящей работе, эти подходы дают один и тот же результат. Однако ряд моментов теории интегрирование достаточно подробно рассмотрен в доступной литературе лишь для интегралов Радона. В целях автономности изложения, удобства ссылок и во избежание разночтений и несогласованности по толкованию  терминов  приведем  некоторые понятия и факты теории интегрирования в духе Н.~Бурбаки \cite{B}, адаптированной на случай  подмножеств в конечномерном пространстве.   

Пусть $k\in \NN$, $Z$  --- локально компактное подпространство  в $\RR^k$ с естественной  евклидовой  топологией; $C_0 (Z)$  ---  пространство непрерывных действительных функций на $Z$ с компактным носителем, или $\mathcal K (Z)$ в обозначениях из \cite{B}, а  $\mu$ --- положительная мера Радона на $Z$, т.\,е. положительный линейный функционал на $C_0(Z)$.
Функция $f\colon Z\to \RR_{\pm\infty}$ является $\mu$-интегрируемой на $Z$, если определен \cite[гл.~IV, \S~4]{B}
\begin{equation}
\text{\it интеграл}\quad \mu\bigl(|f|\bigr)=\int |f| \dd \mu\in \RR \,.
\end{equation}

Пусть $p\in \NN$,   $X\subset \RR^p$  --- локально компактное подпространство  в $\RR^p$, 
\begin{equation}\label{s:k}
T=\{\tau_z\colon z\in Z\}\quad \text{\it--- семейство положительных мер Радона на\/ $X$.} 
\end{equation}
Предположим, что для любой функции $f\in C_0(X)$ функция $z\mapsto \tau_{z}(f)$
является $\mu$-интегрируемой на $Z$. 
Положительный линейный функционал $\nu$ на $C_0(X)$, определенный по правилу \cite[гл.~V, \S~3]{B}, 
\cite[Введение, \S~1]{L}
\begin{equation}\label{semtmu}
\nu(f):=\int \tau_z(f) \dd \mu(z)\, ,
\end{equation} 
--- {\it интеграл семейства мер\/} $T$ из \eqref{s:k} {\it по мере\/ $\mu$}.  Обозначается это  как  
\begin{equation}\label{intmm}
\nu=\int\tau_{z}\dd \mu(z).
\end{equation}

В принятых в этом пп. \ref{imu}  обозначениях в несколько ослабленной, но достаточной для наших применений  форме справедливо далеко идущее развитие и обобщение теоремы  Фубини: 
\begin{Thopi}[{\cite[гл.~V, \S~3, п.~4, теорема 1]{B}}] 
		Пусть  семейство мер $T$ из \eqref{s:k}  удовлетворяет следующим условиям:
	\begin{enumerate}
		\item\label{1mu} Для любой функции $f\in C_0(X)$ функция $z\mapsto \tau_z(f)$ $\mu$-инт\-е\-г\-р\-и\-р\-уема.
		\item\label{2mu} Отображение  $z\mapsto \bigl\{\tau_z(f)\bigr\}_{f\in C_0(X)}$ пространства $Z$  в $\RR^{C_0(X)}$ измеримо.
	\end{enumerate}
Пусть 	$F\colon X\to \RR_{\pm\infty}$ --- $\nu$-интегрируемая функция и для $\mu$-почти всех  $z\in Z$ функция $F$  является $\tau_z$-интегрируемой, а для каждой функции $f\in  C_0(X)$ отображение $z\mapsto \tau_z(f)$ непрерывно. Тогда функция $z\mapsto \tau_z(F)$ является $\mu$-интегр\-и\-р\-у\-е\-м\-ой и 
			\begin{equation}
			\int_X F(x)\dd \nu (x)=\int_Z \left(\int_X F(x)\dd \tau_z(x)\right) \dd \mu (z).
			\end{equation}   
\end{Thopi}	
Здесь мы используем теорему о повторных интегралах в частном случае, когда  $k=p=2$,  $\RR^2$ отождествлено с $\CC$ и $Z=X=\CC$, а в роли семейства $T$ из \eqref{s:k} рассматривается семейство выметаний мер Дирака 
на систему лучей $S$.

\begin{thm}\label{th:IB} Пусть
	$S$ --- замкнутая система лучей  	и для  заряда $\nu\in \mathcal{M}(\CC)$ 
	 выполнено условие Бляшке вне системы лучей $S$, а  $\delta_z^{\bal}$ и $\nu^{\bal}$ --- выметания на $S$  соответственно меры Дирака\/ $\delta_z$ и заряда\/ $\nu$. Тогда
\begin{enumerate}[{\rm (i)}]
	\item\label{dei} для любого подмножества $B\in \mathcal{B}(\CC)$
	\begin{equation}\label{1balo}
	\int \mathbf{1}_B (z')\dd \delta_z^{\bal}(z')\overset{\eqref{chX}}{=}\omega(z,B;\CC\setminus S) \quad\text{при всех $z\in \CC$}, 
	\end{equation}
	т.\,е. $\delta_z^{\bal}=\omega(z,\cdot \,;\CC\setminus S)$ для  всех $z\in \CC$;
	\item\label{deii} для любой функции $f\in C_0(\CC)$ функция $z\mapsto \delta_z^{\bal}(f)$ непрерывна на $\CC$;
		\item\label{deiv} для любой $\nu^{\bal}$-интегрируемой функции $F\colon \CC\to \RR_{\pm\infty}$
\begin{subequations}\label{c:del}
	\begin{align}
\int F(z)\dd \nu^{\bal}(z)&=\int  \mathcal P_{\CC\setminus S}F(z)\dd \nu (z), 
			\tag{\ref{c:del}a}\label{{c:del}a}\\
\intertext{где подынтегральная функция в правой части \eqref{{c:del}a}}
\mathcal P_{\CC\setminus S}F(z)&\overset{\eqref{df:PDf}}{:=}\int F(z')\dd \delta_z^{\bal}(z')
\tag{\ref{c:del}b}\label{{c:del}b}					
	\end{align}
\end{subequations}
$\nu$-интегрируема и называется далее интегралом Пуассона функции $F$ на дополнении $\CC\setminus S$ системы лучей\/ $S$\/ {\rm (ср. с \eqref{df:PDf})};
	\item\label{deiii} выметание  $\nu^{\bal}=(\nu^+)^{\bal}-(\nu^-)^{\bal}$ --- разность выметаний на $S$  верхней и нижней вариаций $\nu$, равных интегралу семейства мер $\delta_z^{\bal}$ по $\nu^{\pm}$, т.\,е.
	\begin{equation}\label{eqomd}
	(\nu^{\pm})^{\bal}\overset{\eqref{intmm}}{=}\int \delta_z^{\bal} \dd \nu^{\pm}(z)=
	\int \omega(z,\cdot \,;\CC\setminus S) \dd  \nu^{\pm}(z).  
	\end{equation}
\end{enumerate}
\end{thm}
\begin{proof} \eqref{dei}.  Следует из очевидных равенств
	\begin{equation*}
\omega(z,B;\CC\setminus S)=\int \omega(z',B;\CC\setminus S) \dd \delta_z(z')
\\\overset{\eqref{df:bal0upS}}{=}\delta_z^{\bal}(B)=\int \mathbf{1}_B (z')\dd \delta_z^{\bal}(z').
	\end{equation*}
\eqref{deii}. Пусть $f\in C_0(\CC)$. Полагаем 
\begin{equation}\label{dfFF}
F(z):=\int f(z')\dd \delta_z^{\bal}(z')
\end{equation}	
Необходимо доказать непрерывность функции $F$. По определению гармонической меры и выметания функция $F$  из \eqref{dfFF} является гармоническим продолжением функции $f$ с границы каждого дополнительного к $S$ угла внутрь этого угла. Из известных классических фактов функция $F$ непрерывна в замыкании каждого дополнительного к $S$ угла и совпадает с $f$ на $S$. Остается показать, что для последовательности точек $(z_n)_{n\in \NN}$, стремящейся к точке $z_0\in S$ и попадающей при этом в разные, вообще говоря,  дополнительные к $S$ углы $\angle_*(\alpha_n,\beta_n)$,   по-прежнему   $\lim\limits_{n\to\infty}F(z_n)=f(z_0)$. Для произвольного числа  
$d>0$ выберем  достаточно малое число $r>0$ так, что 
\begin{equation}\label{drM}
\sup_{z\in D(z_0,2r)}\bigl|f(z)-f(z_0)\bigr|\leq d; \quad M:=\sup_{z\in \CC}\,\bigl|f(z)\bigr|.  
\end{equation}
Для некоторого номера $n_d\in \NN$ при $n\geq n_d$  все точки $z_n$ лежат в $D(z_0,r)$, а граница каждого угла $\angle_*(\alpha_n,\beta_n)$ пересекает границу круга $D(z_0,r)$.  Увеличивая, если это необходимо, номер $n_d$,
можно добиться того, что для некоторого числа $a\in (0,1)$, не зависящего от $n\geq n_d$, выполнены соотношения
\begin{equation}\label{azzn0}
a|z_n|\geq |z| \quad \text{при $z\in {\rm C}_n^0$}; \quad
|z_n|\leq a|z|  \quad \text{при $z\in {\rm C}_n^{\infty}$},
\end{equation}  
где ${\rm C}_n^0$ и ${\rm C}_n^{\infty}$ --- ограниченная и неограниченна связные компоненты множества
$\angle_*(\alpha_n,\beta_n)\setminus D(z_0,2r)$ соответственно. Заметим, что при $z_0=0$ множество  ${\rm C}_n^0$ пусто. Оценим разность 
\begin{equation}\label{dfFFf}
f(z_0)-F(z_n)\overset{\eqref{dfFF}}{=}\int \bigl(f(z_0)-f(z)\bigr)\dd \delta_{z_n}^{\bal}(z).
\end{equation}
Разбивая последний интеграл на сумму интегралов по трем множествам $D(z_0,r)$, ${\rm C}_n^0$ и ${\rm C}_n^{\infty}$, получаем неравенство типа теоремы о двух константах:
 \begin{multline}\label{lm_oma}
\bigl|f(z_0)-F(z_n)\bigr|\overset{\eqref{dfFFf}}{\leq}
\int_{D(z_0,r)} \bigl|f(z_0)-f(z)\bigr| \dd \delta_{z_n}^{\bal}(z)\\
+\left(\int_{{\rm C}^0_n}+\int_{{\rm C}^{\infty}_n}\right) \bigl(|f(z_0)|+|f(z)|\bigr) \dd \delta_{z_n}^{\bal}(z)\overset{\eqref{drM}}{\leq}
d+\int\limits_{{\rm C}^{0}_n\cup {\rm C}^{\infty}_n} 2M\dd \delta_{z_n}^{\bal}(z)\\
\overset{\eqref{1balo}}{=}d+2M \omega\bigl(z_n, {\rm C}^{0}_n\cup {\rm C}^{\infty}_n; \CC \setminus S\bigr)
=d+2M \omega\bigl(z_n, {\rm C}^{0}_n\cup {\rm C}^{\infty}_n; \angle_*(\alpha_n,\beta_n)\bigr)\\
\leq d+2M \omega\bigl(z_n, {\rm C}^{0}_n; \angle_*(\alpha_n,\beta_n)\bigr)+2M \omega\bigl(z_n,  {\rm C}^{\infty}_n; \angle_*(\alpha_n,\beta_n)\bigr),
 \end{multline}
 где последнее равенство следует из определения гармонической меры для дополнения системы лучей 
 из подпункта \eqref{df:harmm} ввиду $z_n\in \angle_*(\alpha_n,\beta_n)$. 
 Покажем, что оба последних слагаемых  в правой части  \eqref{lm_oma} стремятся к нулю при $z_n\to z_0$. 
 Для этого положим 
 \begin{equation}\label{tnozn}
 t_n^0:=\sup_{z\in {\rm C}^{0}_n}|z|, \quad 
 t_n^{\infty}:=\inf_{z\in {\rm C}^{\infty}_n}|z|.
 \end{equation}
  При этом  из соотношений \eqref{azzn0} следует 
 \begin{equation}\label{aznd}
 a|z_n|\geq t_n^0,\quad |z_n|\leq at_n^{\infty} \quad\text{при $n\geq n_d$}.
 \end{equation}
В обозначениях \eqref{tnozn} имеем ${\rm C}^{0}_n\subset \overline{D}( t_n^0)$ и 
${\rm C}^{\infty}_n\subset \CC\setminus D(t_0^{\infty})$, откуда
\begin{multline}\label{lm_omaD}
\bigl|f(z_0)-F(z_n)\bigr|\overset{\eqref{lm_oma}}{\leq}
d+2M \omega\bigl(z_n, {\rm C}^{0}_n; \angle_*(\alpha_n,\beta_n)\bigr)+2M \omega\bigl(z_n,  {\rm C}^{\infty}_n; \angle_*(\alpha_n,\beta_n)\bigr)\\
\leq  d+2M \omega\bigl(z_n, \overline{D}( t_n^0); \angle_*(\alpha_n,\beta_n)\bigr)
+2M \omega\bigl(z_n, \CC\setminus D(t_0^{\infty}); \angle_*(\alpha_n,\beta_n)\bigr)
\end{multline}
По предложению \ref{aalb} при условии $a|z|\geq r$, соответствующему  первому неравенству в \eqref{aznd} 
с $z=z_n$ и $r=t_n^0$,  из \eqref{aalbaa}  получаем 
\begin{multline}\label{Dt}
  \omega\bigl(z_n, \overline{D}( t_n^0); \angle_*(\alpha_n,\beta_n)\bigr)
  \leq \frac{2 (t_n^0)^{\pi/(\beta_n-\alpha_n)}}{\pi\bigl(1-a^{\pi/(\beta_n-\alpha_n)}\bigr)^2}\, 
  \left(-\Im \frac{1}{(z_n)^{\pi/(\beta_n-\alpha_n)}}\right)\\
  \leq \frac{2}{\pi \bigl(1-\sqrt{a}\,\bigr)^2}  \left(\frac{t_n^0}{|z_n|}\right)^{\pi/(\beta_n-\alpha_n)}
  \overset{\eqref{aznd}}{\leq}
  \frac{2}{\pi \bigl(1-\sqrt{a}\,\bigr)^2} \, a^{\pi/(\beta_n-\alpha_n)}
\underset{n\to \infty}{\longrightarrow} 0 \,,
\end{multline}
поскольку $(0,2\pi]\ni \beta_n-\alpha_n\to 0$ при $n\to \infty$ и $0<a<1$.
Аналогично, по предложению \ref{aalb} при условии $ar\geq |z|$, соответствующему  второму неравенству в \eqref{aznd} с $z=z_n$ и $r=t_n^{\infty}$,  из \eqref{aalbaad}  получаем 
\begin{multline}\label{Dti}
\omega\bigl(z_n, \CC\setminus D(t_n^{\infty}); \angle_*(\alpha_n,\beta_n)\bigr)
\leq \frac{2 (t_n^{\infty})^{-\pi/(\beta_n-\alpha_n)}}{\pi\bigl(1-a^{\pi/(\beta_n-\alpha_n)}\bigr)^2}\, 
\Im\, (z_n)^{\pi/(\beta_n-\alpha_n)}\\
\leq \frac{2}{\pi \bigl(1-\sqrt{a}\,\bigr)^2}  \left(\frac{|z_n|}{t_n^{\infty}}\right)^{\pi/(\beta_n-\alpha_n)}  \overset{\eqref{aznd}}{\leq} \frac{2}{\pi \bigl(1-\sqrt{a}\,\bigr)^2} \, a^{\pi/(\beta_n-\alpha_n)}
\underset{n\to \infty}{\longrightarrow} 0 \,.
\end{multline}
Таким образом, из соотношений  \eqref{Dt} и \eqref{Dti} согласно \eqref{lm_omaD} следует неравенство
$\lim\limits_{n\to \infty}\bigl|f(z_0)-F(z_n)\bigr|\leq d$.
Таким образом,   в силу прозвола в выборе числа $d>0$ и точки $z_0\in \CC$ получаем  $\lim\limits_{n\to \infty} F(z_n)=f(z_0)$ для всех $z_0\in \CC$.
Требования \ref{1mu}--\ref{2mu} теоремы о повторных  интегралах в случае, когда в роли мер $\tau_z$  выступают выметания $\delta_z^{\bal}$ мер Дирака достаточно очевидны, что вместе с \eqref{deii} по теореме о повторных интегралах  доказывает \eqref{c:del} из \eqref{deiv}.

Наконец, равенства \eqref{eqomd} п.~\eqref{deiii} сразу следуют из равенства \eqref{1balo}  п.~\eqref{dei} по определению  \ref{df:clbala} после интегрирования обеих частей равенства \eqref{1balo} по верхней и нижней вариациям $\nu^{\pm}$.
\end{proof}

\subsection{Выметание субгармонической функции на систему лучей}
\begin{definition}[{\cite[определение 2.1]{Kha91}, \cite[определение 0.1]{KhDD92}}]
Пусть $u\in \sbh(\CC)$. {\it Выметанием функции $u$ на систему лучей\/} $S$ называем 
функцию $u_S^{\bal}\in \sbh (\CC)$, гармоническую в $\CC \setminus S$ с сужением $u_S^{\bal}\bigm|_S=u\bigm|_S$, т.\,е.  $u_S^{\bal}(z)=u(z)$   для  $z\in S$. Выметание $u_S^{\bal}$ называем также {\it выметанием из дополнительных углов\/} $\CC\setminus S$. 
\end{definition} 
\begin{thm}\label{th:bs}
 Пусть для меры  $\nu\in \mathcal M^+ (\CC)$ выполнено условие   Бляшке вне системы лучей $S$, т.\,е. существует выметание $\nu_S^{\bal}\in \mathcal M^+ (\CC)$ {\rm (теорема \ref{th:Sb})}. 
Тогда для любой функции $u\in \sbh_*(\CC)$ с мерой Рисса $\nu$ существует выметание 
$u_S^{\bal}$ на $S$ с мерой Рисса $\nu_S^{\bal}$.
\end{thm}

Прежде чем  перейти к доказательству теоремы \ref{th:bs}, нам потребуются некоторые  сведения из \cite{Kh07} и \cite[4.1]{HK}.
\subsection{ Субгармонические ядра и глобальное представление Рисса}

Пусть  $B$ --- борелевское подмножество в области  $\Omega$ и $\nu \in \mc M^+(\Omega)$. По определению  $L^1(B,\dd \nu)$ --- множество всех функций $g\colon B \to [-\infty , +\infty]$, интегрируемых по мере $\nu\bigm|_B$, т.\,е. таких, что $\int_B|g| \dd \nu <+\infty$.

\begin{definition}[{\cite[определение 2]{Kh07}}]\label{df:subk} 
Пусть $B$ --- борелевское подмножество в $\Omega$, а борелевская функция $h \colon B\times \Omega \to  \RR$ 
локально ограничена и для каждой фиксированной точки $\zeta \in B$ функция 
$h (\zeta ,\cdot )\colon \Omega \to \RR$ {\it гармоническая на\/} $\Omega$. Тогда функцию
\begin{equation}\label{rep:subk}
k \colon (\zeta , z) \, \longmapsto \, 
\log |\zeta -z|+h (\zeta , z), \quad  (\zeta , z)\in B\times \Omega,	
\end{equation}
будем называть {\it субгармоническим ядром на\/} $B\times \Omega$
({\it  с гармонической компонентой\/ $h$ и несущим множеством\/ $B$}). 
\end{definition}

\begin{definition}[{\cite[определение 3]{Kh07}}]\label{df:subks} 
Пусть мера $\nu \in \mc M^+(\Omega)$
{\it сосредоточена на борелевском подмножестве\/} $B\subset \Omega$. Субгармоническое ядро $k$ на $B\times \Omega$ называем   {\it подходящим для\/} $\nu$,
если для каждой точки $z\in \Omega$ найдутся подобласть $D_z\Subset \Omega$, содержащая точку $z$, и функция 
$g_z \in L^1\bigl((\Omega \setminus D_z)\bigcap B, \dd \nu \bigr)$, для которых  
\begin{equation}\label{dg:nuk}
\sup_{w\in D_z}\bigl| k(\zeta , w )\bigr|\leq g_z(\zeta) \quad \text{\it при всех\/ \; $\zeta \in (\Omega \setminus D_z)\bigcap B$}.
\end{equation}
\end{definition}

Следующее предложение  представляет собой глобальную версию классической теоремы Рисса о локальном представлении субгармонической функции в виде суммы логарифмического потенциала ее меры Рисса и некоторой гармонической функции \cite[Теорема~3.9]{HK}, \cite[Теорема~3.7.9]{Rans}. 
\begin{propos}[{\cite[предложение 3.11]{Kh07}}]\label{pr:Riesz}
Пусть мера\/ $\nu \in \mc M^+(\Omega)$ сосредоточена на борелевском подмножестве\/ $B\subset \Omega$ и субгармоническое ядро\/ $k$ на\/ $B\times \Omega$ --- подходящее для меры\/ $\nu$. 
Тогда интеграл 
\begin{equation}\label{rep:glpot} 
U^{\nu}_k(z):=	\int_B k (\zeta , z) \dd \nu( \zeta ), \quad z\in \Omega,
\end{equation}
определяет субгармоническую на\/  $\Omega$ функцию с мерой Рисса\/ $\nu$,
а для каждой функции $M\in \sbh (\Omega)$ с мерой Рисса\/ $\nu_M=\nu$ имеет место представление
\begin{equation}\label{repr:Riesz}
M=	U_k^{\nu}+H \quad \text{на \; $\Omega$, \; где  $H\in \Har (\Omega)$}.
\end{equation}
\end{propos}

В случае произвольной субгармонической в $\Omega:= \CC$ функции $u\neq \boldsymbol {-\infty}$ с мерой Рисса $\nu:=\nu_u$ конструкция одного  из возможных вариантов подходящего для меры $\nu$ субгармонического ядра давно известно как ядро Вейерштрасса \cite[4.1]{HK}. Для согласованности с обозначениями из \cite{HK} введем его в той же форме, что и в  \cite[4.1]{HK} (в наших  записях переменная  $z$  соответствует переменной $x$, а $\zeta$  --- соответственно 
 $\xi$ или $\zeta$ в \cite[4.1]{HK} и, кроме того, переменные переставлены местами).       
Положим 
\begin{equation}\label{Ko}
K(z):=\log |z|, \quad K(\zeta-z)=\sum_{j=0}^\infty a_j(\zeta, z)=
\Re \Bigl\{\log |\zeta|-\sum_{j=1}^{\infty}\frac1{j}\Bigl(\frac{z}{\zeta}\Bigr)^j\Bigr\}
\end{equation}
--- разложение гармонической по $z$ в $\CC$, исключая точку $z=\zeta$, в степенной ряд по переменным $z,\bar z$, где при фиксированном $j$ и $\zeta\neq 0$ через  $a_j(\zeta, z)$, $j\in \NN$, обозначены одночлены от $z,\bar z$ степени $j$. При этом одночлены $a_j(\zeta, z)$, $j\in \NN$, гармоничны по $z$ при фиксированном $\zeta$ и непрерывны по совокупности переменных $z,\zeta$ при $|\zeta|\neq 0$
\cite[лемма 4.1]{HK}. Наряду с  \eqref{Ko} рассмотрим <<урезанные>> ядра \cite[(4.1.2)]{HK}
\begin{equation}\label{Kqz}
K_q(\zeta,z):= K(\zeta-z)-\sum_{j=0}^q a_j(\zeta, z), \quad q\in \NN_0.
\end{equation}
являющиеся субгармоническими ядрами. 

Предположим пока, что $\supp \nu \cap \DD$. Для  меры $\nu\in \sbh(\CC)$ всегда можно подобрать положительную возрастающую функцию $q\colon \RR_+\to
\{-1\}\cap \NN_0$, непрерывную  справа, тождественно равную $-1$ на отрезке $[0,1]$,  для которой 
\begin{equation}\label{conver}
\int_{0}^{\infty} \Bigl(\frac{t_0}{t}\Bigr)^{q(t)+1} \dd \nu^{\rad}(t)<+\infty.
\end{equation}
Рассмотрим новое ядро \cite[(4.1.10)]{HK}) 
\begin{equation}\label{Kqzz}
 K_{q(|\zeta|)}(\zeta, z), \;\text{(при $|\zeta|\leq 1$ это $\log|\zeta-z|$, поскольку $\sum_{j=0}^{-1}\dots:=0$ в 
\eqref{Kqz})}  
\end{equation}
--- субгармоническое ядро, которое, как легко следует из известных оценок  ядер \eqref{Kqz} \cite[4.1.1]{HK} и интергрирования по частям интеграла \eqref{conver},
подходящее для меры $\nu$. 
В этом случае глобальное представление Рисса \eqref{rep:glpot}  из предложения \ref{pr:Riesz} задается равенством
\begin{equation}\label{rep:glpotU} 
U^{\nu}_{K_{q(|\cdot|)}}(z):=	\int_\CC K_{q(|\zeta|)} (\zeta , z) \dd \nu( \zeta ), \quad z\in \CC,
\end{equation}
а аналог    равенства \eqref{repr:Riesz} задается представлением Вейерштрасса \cite[теорема 4.1]{HK}:
\begin{equation}\label{reprW}
u(z)=U^{\nu}_{K_{q(|\cdot|)}}(z) +H(z)=\int_{\CC} K_{q(|\zeta|)}(\zeta, z) \dd \nu_u(\zeta)+H(z), 
\end{equation}
где  $H\in \Har (\CC)$, а интеграл сходится абсолютно в окрестности точки $\infty$ и равномерно в любом ограниченном круге.

Теперь можем обратиться к доказательству теоремы \ref{th:bs}.
\begin{proof} 
Для достаточно малых чисел $\sigma\in (0,1)$ введем обозначение
\begin{equation*}
\log_{\sigma}t:=\max\{\log t, \log \sigma\}
\end{equation*}
и в присутствующие в ядре \eqref{Kqzz} логарифмические выражения
\begin{equation*}
\log |\zeta-z| \quad\text{и, возможно, появляющееся при $|\zeta|>1$,} \; \log \Bigl|1-\frac{z}{\zeta}\Bigl|
\end{equation*}
заменим на непрерывные
\begin{equation}
\log_{\sigma} |\zeta-z| \quad\text{и} \; \log_{\sigma} |\zeta-z|-\log |\zeta| \quad\text{уже при $|\zeta|>1$.}
\end{equation}
При такой замене новое ядро $K_{q(|\cdot|)}^{\sigma}(z)$ становится кусочно непрерывным и локально ограниченным, а при $\sigma$, стремящемся к нулю, убывая, стремится к \eqref{Kqzz}. Применяя теорему  
\ref{th:IB}\eqref{deiv} при $F=K_{q(|\cdot|)}^{\sigma}(z)$ и устремляя $\sigma$ к нулю, получим
равенство
 \begin{equation*}
\int K_{q(|\cdot|)}(\zeta_1,z) \dd \nu^{\bal}(\zeta)
=\int\Bigl( K_{q(|\cdot|)}(\xi,z) \dd \delta_{\zeta}^{\bal}(\xi)\Bigl )\dd \nu(\zeta),
\end{equation*}
где для всех $z\in S$ 
\begin{equation*}
K_{q(|\cdot|)}(\xi,z) \dd \delta_{\zeta}^{\bal}(\xi)=K_{q(|\cdot|)}(\zeta,z), \quad z\in S.
\end{equation*}
Таким образом, в качестве выметания субгармонической функции $U$ можем положить функцию
\begin{equation*}
U^{\bal}(z):=\int K_{q(|\cdot|)}(\zeta,z) \dd \nu^{\bal}(\zeta).
\end{equation*}
Добавка отброшенного ранее гармоническое слагаемое $h$ дает $u^{\bal}=U^{\bal}+h$, что и нужно.
\end{proof}

Приведенное здесь доказательство теоремы \ref{th:bs} технически упрощает доказательство  \cite[основная терема]{Kha91}, где был рассмотрен лишь случай функций конечного порядка и позволяет не выделять как особый незавершенный случай ситуации функций и мер порядка $\rho<1$. Теперь основной результат \cite[основная терема]{Kha91} можно сформулировать в законченной форме:
\begin{thm}\label{mthfo}
Пусть $p\in \RR^+_*$, $u\in \sbh_*(\CC)$ и $\type_p^{\infty}[u]<+\infty$, S --- допустимая система лучей для меры Рисса $\nu_u$, а $\nu_u^{\bal}$ --- выметание меры  $\nu$ на $S$. Тогда существует выметание $u^{\bal}\in \sbh_*(\CC)$ с  $\type_p^{\infty}[u^{\bal}]<+\infty$ и мерой Рисса $\nu^{\bal}$ с $\type_p^{\infty}[\nu^{\bal}]<+\infty$, удовлетворяющей при целом $p$ условию Линделефа 
{\rm (подраздел~\ref{Lc}, определение \ref{con:Lp}, теорема \ref{th:Lc})}.
\end{thm}

\section{Приложения классического выметания на систему лучей}\label{applA} 
\setcounter{equation}{0}

\subsection{ Полная регулярность роста субгармонической или целой функции на системе лучей}\label{vpreg}

 Определения полной регулярности роста целой и/или субгармонической функции и правильно распределенной последовательности нулей и/или меры можно найти в \cite{Levin56}, \cite{Az}).
Напомним их  здесь.

Пусть $p\in \RR_*^+$. Функция $u\in \sbh_*(\CC)$ c $\type_p^{+\infty}[u]<+\infty $ {\it вполне регулярного роста\/} (при порядке $p$) на луче $l_{\theta}:=\{re^{i\theta}\colon r\in \RR^+\}$, если существует предел
\begin{equation*}
\lim_{\substack{r\to +\infty\\r\notin E}} \frac{u(re^{i\theta})}{r^p}\, , \quad  \lim_{t\to \infty}\frac{\lambda_{\RR}\bigl(E\cap  [-t,t]\bigr)}{t}=0
\end{equation*}
 т.\,е. когда $r\to +\infty$, не принимая значений из некоторого множества $E$ нулевой относительной линейной меры. Целая функция $0\neq f\in \Hol(\CC)$ с $\type_p^{\infty}[\log|f|]<+\infty$ вполне регулярного роста (при порядке $p$) на луче $l_{\theta}$, если таковой является функция $u=\log|f|$. Функция вполне регулярного роста на системе лучей $S$ (при порядке $p$) если она конечного типа при порядке $p$ около бесконечности и вполне регулярного роста на каждом луче из $S$.

Мера $\nu\in \mc M^+(\CC)$ с $\type_p^{\infty}[\nu]<+\infty$ {\it правильно распределена на $\CC$} при нецелом порядке $p$, если при всех $\alpha<\beta\in \RR\setminus N$, где множество $N$ не более чем счетно, существует конечный предел {\it (угловая плотность)}
\begin{equation}\label{ugpl}
b_{\nu}(\alpha, \beta)=\lim_{r\to +\infty} \frac{\nu\bigl(\angle [\alpha, \beta] \bigr)}{r^p} .
\end{equation} 
При целом $p$ для правильной распределенности меры $\nu$ дополнительно требуется существование конечного предела 
\begin{equation}\label{upL}
\lim_{r\to +\infty}\int_{1\leq r\leq r} \frac{1}{z^p} \dd \nu(z) \quad\text{\it --- правильное условие Линделефа}.
\end{equation}

В \cite[теорема 4.1]{Kha91}, \cite[теорема 1.5.1]{KhDD92} дано основное приложение выметания 
на системе лучей в предположении, что в случае системы лучей $S$, состоящей из одного луча порядок $p<1$. По той же схеме, что и при доказательстве упомянутых в предыдущей фразе теорем, теорема \ref{mthfo} позволяет снять это ограничение. 
\begin{thm}[{\rm \cite[теорема 4.1]{Kha91}, \cite[теорема 1.5.1]{KhDD92}}]\label{crgr}
Пусть $p\in \RR^+_*$ и   $S$ --- система лучей, допустимая для меры Рисса $\nu_u$ функции $u\in \sbh_*(\CC)$ с 
$\type_p^{\infty}[u]<+\infty$. Функция $u$ вполне регулярного роста на $S$, если и только если выметание $\nu_u^{\bal}$ на $S$ --- правильно распределенная  мера на $\CC$. 
\end{thm}
 Приведем здесь конкретные реализации теоремы \ref{crgr} в виде двух примеров с явными соотношениями, приведенных только в диссертации \cite{KhDD92}, примеры 1.5.1, 1.5.2, в неполной и недостаточно корректной  форме.  
\begin{example}[{\rm \cite[теорема 1.5.2]{KhDD92}}]\label{exgr}  Пусть $f\in \Hol_*(\CC)$ --- целая функция конечного типа при порядке 1 и класса Ахиезера, или класс A, \cite{Levin56} с последовательностью нулей ${\tt Z}:=\Zero_f=\{{\tt z}_k\}_{k=1,2,\dots}$, что по определению означает выполнение условия Бляшке определения  \ref {c:akh} около бесконечности для  верхней и нижней полуплоскостей $\CC^{\up}$ и $\CC_{\lw}$ одновременно:
\begin{equation}\label{imz}
\sum_{|{\tt z}_k|\neq 0} \Bigl|\Im \frac{1}{{\tt z}_k}\Bigr|<+\infty,
\end{equation} 
для системы двух лучей $S=\RR=(-\RR^+)\cup \RR^+$,
или , в эквивалентной форме, условие Ахиезера 
\begin{equation}\label{Af}
\limsup_{r\to +\infty}\int_1^r\frac{\log \bigl|f(x)f(-x)\bigr|}{x^2} \dd x <+\infty
\end{equation}
Условие \eqref{imz} в точности означает, что вещественная ось допустима для считающей меры $n_{\tt Z}$
\eqref{df:nZS}. По определению выметания меры и геометрическому смыслу гармонической меры   
выметание $n_{\tt Z}^{\bal}$ на $\RR$ определяется по правилу
\begin{equation*}
n_{\tt Z}^{\bal}\bigl([t_1,t_2]\bigr)=\sum_k\omega \bigl({\tt z}_k, [t_1, t_2]\bigr),
\end{equation*}
где $\omega \bigl({\tt z}_k, [t_1, t_2]\bigr)$ --- угол, под которым виден отрезок $[t_1,t_2]$ из точки ${\tt z}_k$, деленный на $\pi$. Рассмотрим, как  в \eqref{nuR}, функцию распределения $(\nu_{\tt Z}^{\bal})^{\RR}$ выметенной меры 
$\nu_{\tt Z}^{\bal}$ с носителем на $\RR$. 
По теореме \ref{crgr} {\it целая функция экспоненциального типа $f$ класса A вполне регулярного роста на $\RR$ тогда и только тогда, когда существуют три  предела}
\begin{equation*}
\lim_{t\to \pm\infty}  \frac{(\nu_{\tt Z}^{\bal})^{\RR}(t)}{t} ,
\quad \lim_{t\to+\infty}\int_1^t\frac{1}{s}\bigl((\nu_{\tt Z}^{\bal})^{\RR}(s)+(\nu_{\tt Z}^{\bal})^{\RR}(-s)\bigr)
\dd s.
\end{equation*}
Здеcь существование двух первых пределов означает существованию плотностей при движении вправо и влево 
соответственно к $\pm \infty$ (ср. с \eqref{ugpl}), а существование последнего предела соответствует существованию предела в правильном условии Линделефа \eqref{upL}.
\end{example}

\begin{example}[{\rm \cite[теорема 1.5.1]{KhDD92}}]\label{exgr} Пусть $f$ --- целая функция экспоненциального типа с нулями на лучах 
\begin{equation}
l_k:=\Bigl\{ t\exp i\Bigl(\frac{\pi}{4}+\frac{\pi k}{2}\Bigr)\colon t\geq 0\Bigr\}, \quad k=0,1,2,3,
\end{equation}
т.\,е. на биссектрисах четвертей координатной плоскости, $n_k(t)\geq 0$ --- функции распределения \eqref{nuR}, или считающие радиальные функции, нулей, лежащих на лучах $l_k$, $t\geq 0$. Полагаем $n_4(t)\equiv n_0(t)$. 
Функция $f$ вполне регулярного роста при порядке $1$ одновременно на вещественной и мнимой осях тогда и только тогда, когда существуют конечные пределы
\begin{subequations}\label{densn}
\begin{align}
\lim_{t\to+\infty} 2\int_0^{+\infty}\frac{n_k(s)+n_{k+1}(s)}{s^4+t^2}& \, s \dd s=b_k, \quad k=0,1,2,3,
\tag{\ref{densn}d}\label{densnd}\\
\lim_{r\to+\infty}\int_1^r\int_0^{+\infty}\Bigl( \sum_{k=0}^3 i^{k+1}\bigl( n_k(s)&+n_{k+1}(s)\bigr)\Bigr)
\frac{s \dd s \dd t}{t(s^4+t^2)}.
\tag{\ref{densn}L}\label{densnL}
\end{align}
\end{subequations}
При этом индикатор $h_f$ функции $f$таков, что 
\begin{equation*}
h_f\Bigl(\frac{\pi}{2}\Bigr)+h_f\Bigl(-\frac{\pi}{2}\Bigr)=b_0+b_2,
\quad 
h_f(0)+h_f(\pi)=b_1+b_3.
\end{equation*} 
Существование пределов \eqref{densnd} соответствует существованию угловой плотности (ср. с 
\eqref{ugpl}), а существование предела  \eqref{densnL} --- правильному условию Линдедефа \eqref{upL}.

\end{example}
\subsection{ Выметание на $\RR$ целых функций класса A, мультипликаторы, полнота систем экспонент}\label{bAmp} В качестве системы лучей $S$ в этом подразделе рассматриваем вещественную ось $\RR$, дополнительные углы для которой --- верхняя и нижняя полуплоскости $\CC^{\up}$ и $\CC_{\lw}$. 
\begin{thm}\label{thmA} Пусть $u\in \sbh_*(\CC)$ с мерой Рисса $\nu_u$ конечного типа при порядке $1$ около $\infty$, т.\,е. $\type_p^{\infty}[u]<+\infty$, для меры $\nu_u$ выполнено условие Бляшке около $\infty$ одновременно для верхней и нижней полуплоскостей, т.\,е.
\begin{equation}\label{coIm}
\int_{\CC\setminus \DD}\Bigl|\Im \frac{1}{z}\Bigr|\dd \nu_u(z)<+\infty,
\end{equation}
что в терминах самой функции $u$ эквивалентно условию {\rm (см. \cite[лемма 3.2]{Kha91}, \cite[лемма 1.4.3]{KhDD92}, ср. с \eqref{Af})}
\begin{equation}\label{coImu}
\limsup_{r\to +\infty}\int_1^{r}\frac{u(t)+u(-t)}{t^2} \dd t<+\infty
\end{equation}
Тогда для любого числа $\e>0$ найдется целая функция экспоненциального типа $g$ с нулями только на вещественной оси, с которой для индикатора $h_{u+\log |g|}$ выполнена
оценка  
\begin{equation}\label{uwout}
h_{u+\log|g|}(0)+h_{u+\log |g|}(\pi)\leq \e,
\end{equation}
\end{thm}
\begin{proof} Будет использована 
\begin{lemma}[{\rm частный случай основной теоремы из \cite{Kha01l}}] Пусть число $\beta\in (0,\pi/2)$ и $\mu$ --- мера конечной верхней плотности
\begin{equation}\label{bnuf}
b_{\mu}\overset{\eqref{ugpl}}{:=}\limsup_{t\to +\infty}\frac{\mu^{\rad}(t)}{t}<+\infty,
\end{equation}
 носитель которой содержиться в паре вертикальных углов 
\begin{equation}\label{vertug}
\Bigl\{ z:\bigl|\arg z-\frac{\pi}2\bigr|\le \beta \Bigr\}
\cup \Bigl\{ z:\bigl|\arg z+\frac{\pi}2\bigr|\le \beta \Bigr\},
\end{equation}
Тогда для любой субгармонической функции $v_{\mu}$ с мерой Рисса $\mu$ при условии 
\begin{equation}\label{O1Re}
\biggl| \; \int\limits_{1<|\zeta |\le r} \Re \frac1{\zeta}\, d\mu (\zeta )
\biggr| =O(1)\, , \quad r\to +\infty \, 
\end{equation}
найдется целая функция $\tilde g$ с нулями только на мнимой оси, для которой сумма $v_\mu+\log|\tilde g|$ --- субгармоническая функция конечного типа при порядке $1$ с индикатором $h_{v_{\mu}+\log |\tilde g|}$, удовлетворяющим оценке 
\begin{equation}\label{v+lg}
h_{v_{\mu}+\log |\tilde g|}\Bigl(\pm \frac{\pi}2\Bigr) <
\frac{12\pi (\pi +2){b}_{\mu}}{{\pi}/2-\beta}
\cdot\tg \beta \, ,
\end{equation}
\end{lemma}
Для удобства применения леммы <<повернем>> выметенную на $\RR$ функцию $u^{\bal}$ на $\pi/2$ радиан <<против часовой стрелки>>, т.\,е. применим замену $v(z):=u^{\bal}(ze^{-i\pi/2})$, $z\in \CC$. Тогда по 
теореме \ref{mthfo} функция $v$ конечного типа при порядке $1$ с носителем ее меры Рисса $\supp \mu\in i\RR$
 и  тем более удовлетворяет условию \eqref{O1Re} с постоянной $0$ в правой части вместо $O(1)$. Всегда можно выбрать число $\beta>0$ cтоль малым,  что правая часть в \eqref{v+lg} не больше  $\e>0$ из условия теоремы \ref{thmA}. Обратный поворот <<по часовой стрелке>>  показывает, что для некоторой целой функции экспоненциального типа $g$ (<<повернутая функци $\tilde g$>>) с нулями только на вещественной оси  выполнена
оценка  
\begin{equation*}
h_{u^{\bal}+\log|g|}(0)+h_{u^{\bal}+\log |g|}(\pi)\leq \e,
\end{equation*}
откуд ввиду совпадения $u=u^{\bal}$ на $\RR$ получаем требуемое \eqref{uwout}.
\end{proof}
\begin{remark} В теореме \ref{thmA} невозможно добиться ограниченности функции $u+\log |g|$ на $\RR$, так как в этом случае из субгармонического аналога легкой необходимой части теоремы Берлинга\,--\,Мальявена  о мультипликаторе функция $u\in \sbh_*(\CC) $ обязана быть функцией класса Картрайт \cite{MS}, \cite{BTKh}, , т.\,е. вместо условия \eqref{coImu}  
должен сходиться интеграл
\begin{equation*}
	J_2^+[u]\overset{\eqref{J2+}}{=}\int_{-\infty}^{+\infty}\frac{u^+(t)}{1+t^2} \dd t<+\infty,
\end{equation*}
что гораздо более сильное требование, чем \eqref{coImu}  \cite[теорема 12]{Levin56}.
\end{remark}
\begin{corollary}\label{cormult} Для любой целой функции $f$ экспоненциального типа класса A  при любом $\e>0$ существует целая функция функция-мультипликатор $g$ экспоненциального типа только с вещественными нулями, с которой  произведение $fg$ ---целая функция экспоненциального типа  и индикаторной диаграммой, лежащей внутри вертикальной полосы  ширины не больше $\e$. 
\end{corollary}
\begin{proof} Применить теорему \ref{thmA} к функции $u=\log |f|$.
\end{proof}
\begin{corollary}\label{corcomp} Для любой последовательности $\tt Z=\{{\tt z}_k\}$ конечной верхней плотности, удовлетворяющей условию  Бляшке  \eqref{imz} около $\infty$ для   $\CC^{\up}$ и $\CC_{\lw}$ одновременно и для любого числа $\e>0$, экспоненциальная система 
\begin{equation}
\Exp^{\tt Z}:=\Bigl\{ z^m\exp({\tt z}_kz)\colon 0\leq m\leq n_{\tt Z}\bigl(\{{\tt z}_k\}\bigr)-1\Bigr\}
\end{equation}
не полна в любом пространстве $\Hol (S_{\e})$ в топологии равномерной сходимости на компактах, когда  $S_{\e}$
--- вертикальная (полу)полоса ширины $\e$.   
\end{corollary}
\begin{proof} Сразу следует из следствия \ref{cormult} b стандартной классической взаимосвязи между последовательностями неединственности для целых функций экспоненциальных фунций и неполнотой экспоненциальных систем в пространствах голоморфных функций ( см. последние параграфы в \cite{Kha89}--\cite{kh91AA} применительно к (полу)полосам и многочисленные часто употребляемые схемы в \cite{Khsur}). 
\end{proof}

\end{document}